\documentclass[12pt,a4paper]{article}
 
\usepackage[latin1]{inputenc}
\usepackage[english]{babel}
\usepackage[T1]{fontenc}
\usepackage{amsmath} 
\usepackage{amsfonts}
\usepackage{amssymb}
\usepackage{amsthm}
\usepackage{ dsfont }
\usepackage{ stmaryrd }
\usepackage{tikz}
\usetikzlibrary{matrix,arrows.meta}
\usepackage[title]{appendix}
\usepackage{todonotes}

\allowdisplaybreaks

\usepackage{hyperref}

\usepackage{caption}

\usepackage{graphicx,xcolor} 
\usepackage[a4paper,top=2.1cm,bottom=2.60cm,left=2.1cm,right=2.1cm]{geometry} 

\theoremstyle{plain}
\newtheorem{theorem}{Theorem}[section]
\newtheorem{lemma}[theorem]{Lemma}
\newtheorem{proposition}[theorem]{Proposition} 
\newtheorem{corollary}[theorem]{Corollary}

\theoremstyle{definition} 
\newtheorem{remark}[theorem]{Remark}
\newtheorem{definition}[theorem]{Definition}

\begin{document}

\title{Shape perturbation of Grushin eigenvalues}

\date{September 2, 2020 - 20200902\_shapegrushin}

\author{Pier Domenico Lamberti\thanks{Dipartimento di Matematica ``Tullio Levi Civita'', Universit\`a degli Studi di Padova, via Trieste 63, 35121 Padova , Italy. Email: lamberti@math.unipd.it}, Paolo Luzzini\thanks{Dipartimento di Matematica ``Tullio Levi Civita'', Universit\`a degli Studi di Padova, via Trieste 63, 35121 Padova , Italy. Email: pluzzini@math.unipd.it}, Paolo Musolino\thanks{Dipartimento di Scienze Molecolari e Nanosistemi, Universit\`a Ca' Foscari Venezia, via Torino 155, 30172 Venezia Mestre, Italy. Email: paolo.musolino@unive.it}}

\maketitle

\noindent
{\bf Abstract:} We consider the spectral problem for 
the Grushin Laplacian subject to homogeneous Dirichlet boundary conditions on a 
bounded open subset of $\mathbb{R}^N$.  We prove that the symmetric functions of the eigenvalues depend
real analytically upon domain perturbations and we prove an  Hadamard-type formula for their shape differential. 
In the case of perturbations depending on a single scalar parameter, we prove a Rellich-Nagy-type theorem which describes the bifurcation phenomenon of multiple eigenvalues.
As corollaries, we characterize the critical shapes under isovolumetric and isoperimetric perturbations in terms of overdetermined problems and we deduce a new proof of the Rellich-Pohozaev identity for the  Grushin eigenvalues.

\vspace{9pt}

\par
\vspace{9pt}



\section{Introduction}

In this paper we consider the following degenerate elliptic operator in $\mathbb{R}^N$:
\begin{equation*}\label{Gdef}
\Delta_{G} := \Delta_x +|x|^{2s}\Delta_y, \qquad s \in \mathbb{N}. 
\end{equation*}
Here and throughout the paper  $N\in \mathbb{N}$, $N \geq 2$, $h, k \in \mathbb{N}$, 
$h+k=N$, $x \in \mathbb{R}^h$, $y \in \mathbb{R}^k$, where $\mathbb{N}$ denotes the set of positive integers. The vector $x$ denotes the first $h$ components 
of a vector $z \in \mathbb{R}^N$, and similarly $y$ denotes the last $k$ ones, {\it i.e.} $z = (x,y) \in \mathbb{R}^h\times\mathbb{R}^k=\mathbb{R}^N.$ 
By $\Delta_x$ and $\Delta_y$ we denote the standard Laplacians with respect to the $x$ and $y$ variables, respectively. 
The operator $\Delta_G$ is nowadays known as the Grushin Laplacian, and has been introduced in a preliminary version by Baouendi \cite{Ba67} and Grushin  \cite{Gr70, Gr71}.  In  \cite{Ba67}, Baouendi 
has studied the regularity of the solutions of a boundary value problem for an elliptic operator, whose coefficients may vanish on the boundary of the open set where the problem is considered. In  \cite{Gr70, Gr71}, Grushin has considered a class of operators that degenerate on a submanifold. Later on, a more general notion of the Grushin Laplacian has been introduced and studied by
Franchi and Lanconelli \cite{FrLa82, FrLa83, FrLa84}. 
In recent years, these operators have been studied under several points of view.  Here we mention just a few contributions, without the aim of completeness. For example, inequalities and estimates related to the Grushin operator have been investigated by many authors. D'Ambrosio \cite{Da04} has studied Hardy inequalities related to Grushin-type operators.   Garofalo and Shen \cite{GaSh94} have obtained Carleman estimates  and unique continuation for the Grushin operator. Symmetry, existence and uniqueness properties of extremal functions for the weighted Sobolev inequality are obtained in Monti \cite{Mo06}. Monticelli, Payne, and Punzo \cite{MoPaPu19} have obtained Poincaré inequalities for Sobolev spaces with matrix-valued weights with applications to the existence and uniqueness of solutions to linear elliptic and parabolic degenerate partial differential equations. Furthermore, several authors have investigated issues related to the solutions to problems for degenerate equations.  Kogoj and Lanconelli  have proved in \cite{KoLa09} a Liouville theorem for a class of linear degenerate elliptic operators, whereas in \cite{KoLa12} they have obtained some existence, nonexistence and regularity results for boundary value problems for semilinear degenerate equations. Monticelli \cite{Mo10}  has obtained a maximum principle for a class of linear degenerate elliptic differential operators of the second order. Thuy and Tri \cite{ThTr02, ThTr12} and Tri \cite{Tr98A, Tr09} have analyzed boundary value problems for linear or semilinear degenerate elliptic differential equations. Other references can be found in the survey of Kogoj and Lanconelli \cite{KoLa18}, where the authors have discussed linear and semilinear problems involving the $\Delta_{\lambda}$--Laplacians, which contain, as a particular case, the operator introduced by Baouendi and Grushin.

In our work we are interested in the eigenvalue 
problem
\begin{equation}\label{Geigp}
-\Delta_{G} u = \lambda  u,
\end{equation}
with  zero Dirichlet boundary conditions on a  variable bounded open subset $\Omega$ 
of $\mathbb{R}^N$.
It is well known that problem \eqref{Geigp} admits a divergent 
sequence of domain dependent eigenvalues of finite multiplicity:
\[
0 < \lambda_1[\Omega] \leq \cdots \leq \lambda_n[\Omega] \leq \cdots \to +\infty.
\]
Our main aim is to understand the dependence of the eigenvalues $\lambda_n[\Omega]$,  both simple and multiple, upon perturbation of the domain $\Omega$. In particular, 
we plan to extend the results of Lamberti and Lanza de Cristoforis    \cite{LaLa04} for 
the Laplacian and of Buoso and Lamberti \cite{BuLa13, BuLa15} for polyharmonic operators and systems
 to the case of the Grushin Laplacian $\Delta_G$.

Shape sensitivity analysis and shape optimization of quantities and functionals related to partial differential equations   are vast topics which have been investigated  by several authors with different techniques.  We 
 mention, for example,  the monographs by Bucur and Buttazzo \cite{BuBu05},  Daners \cite{Da08}, Delfour and Zol\'esio \cite{DeZo11}, Henrot \cite{He06}, Henrot and Pierre \cite{HePi05},  Novotny and Soko\l owski \cite{NoSo13},   Novotny, Soko\l owski, and \.{Z}ochowski \cite{NoSoZo19}, Pironneau \cite{Pi84}, and    Soko\l owski and Zol\'esio \cite{SoZo92}. One of the central problems concerns the 
 analysis of the dependence of the eigenvalues of partial differential operators upon domain perturbations. 
  Many authors studied the qualitative behavior of the eigenvalues 
 of various partial  differential operators with respect to shape perturbations proving, for example, continuity, smoothness or even analyticity results. In addition to the above monographs,
 we mention in this direction the works of  Arendt and Daners \cite{ArDa08}, Arrieta \cite{Ar95}, Arrieta and Carvalho
 \cite{ArCa04},  Buoso and Lamberti \cite{BuLa13, BuLa15},  Buoso and Provenzano \cite{BuPr15}, Fall and Weth \cite{FaWe19},
  Lamberti and Lanza de Cristoforis   \cite{LaLa04}, and Prodi \cite{Pr94}.  These issues are closely related to the shape optimization of  
  eigenvalues. Indeed, a first step towards the maximization or minimization of an eigenvalue under suitable constraints (such as fixed volume or perimeter) is to study critical shapes. Accordingly, a detailed analysis of the 
  regularity upon shape perturbations and of the shape differential is crucial for this kind of optimization problems. 
The problem of minimizing the first eigenvalue of the Dirichlet Laplacian has been solved by  
Faber \cite{Fa23} and Krahn  \cite{Kr24}, and later on other authors have generalized their result to different operators (see, 
{\it e.g.}, Ashbaugh and Benguria \cite{AsBe95} and Nadirashvili \cite{Na95}). However,  in general, finding the shapes which optimize a certain eigenvalue is a hard problem which remains open for several well-studied operators, including  the Grushin Laplacian.
  
Another point of view in spectral shape sensitivity analysis is proving quantitative stability estimates 
  for the eigenvalues in terms of some notion of vicinity of sets. For this topic, which is outside the scope of the present work, we refer to the survey of Burenkov, Lamberti and Lanza de Cristoforis \cite{BuLaLa06}.

This paper is in the spirit of studying the qualitative behavior of the eigenvalues of the Grushin Laplacian upon shape perturbations. Namely, in contrast with other approaches in the literature which  address only continuity and differentiability issues, in Theorem \ref{thm:realanalsym} we prove that the symmetric functions of the eigenvalues 
depend real analytically upon shape perturbations. We note that considering the symmetric functions of the eigenvalues, and not the eigenvalues themselves, is a  natural choice. Indeed, a perturbation of the domain can 
split a multiple eigenvalue into different eigenvalues of lower multiplicity and thus the corresponding
branches can have a corner at the splitting point.   Furthermore, we obtain the Grushin analog of the Hadamard formula for the shape differential (see formula \eqref{Nhadf1} of Theorem \ref{Nhadf}).  In the case of perturbations depending real 
analytically on a single scalar parameter $\varepsilon$, we prove a Rellich-Nagy-type theorem which describes 
the bifurcation phenomenon of multiple eigenvalues that we mentioned before. More precisely, given 
 an eigenvalue $\lambda$ of multiplicity $m$ on $\Omega$ and a family of  perturbations
$\{\phi_\varepsilon\}_{\varepsilon \in \mathbb{R}}$ of $\Omega$ depending real analytically on $\varepsilon$
and such that $\phi_0$ is the identity, our 
result guarantees that all the branches splitting from $\lambda$ at $\varepsilon = 0$ are described by $m$ real 
analytic functions of $\varepsilon$.  Moreover, the right derivatives at $\varepsilon = 0$ of the branches splitting 
from $\lambda$ coincide with the eigenvalues of the matrix
\begin{equation}\label{form:eigenmatrix}
\left( -\int_{\partial\Omega}  \left(\frac{d }{d\varepsilon}\phi_\varepsilon\Bigr|_{\varepsilon = 0} \mathbf{n}^t\right)  \frac{\partial v_i}{\partial \mathbf{n}}\frac{\partial v_j}{\partial \mathbf{n}}|\mathbf{n}_G|^2\,d\sigma \right)_{i,j= 1,\ldots, m},
\end{equation}
where $\{v_i\}_{j=1,\ldots,m}$ is an orthonormal basis in $L^2(\Omega)$ of the eigenspace corresponding to 
$\lambda$, $\mathbf{n}$ is the outer unit normal field to $\partial\Omega$, and 
$\mathbf{n}_G := (\mathbf{n}_x, |x|^s\mathbf{n}_y)$ (see formula \eqref{RN1B}). 
We note that formula \eqref{form:eigenmatrix} and the analogous formulas of the paper based on surface integrals are obtained by assuming that the eigenfunctions are of class $W_0^{1,2}(\Omega)\cap W^{2,2}(\Omega)$ or at least of class  $W_0^{1,2}(U)\cap W^{2,2}(U)$ for some neighborhood $U$ of the support of the perturbation, here $\frac{d }{d\varepsilon}\phi_\varepsilon\Bigr|_{\varepsilon = 0}$. This assumption is clearly automatically guaranteed by classical regularity theory when  $\overline{U}$ does not intersect the set $\{x=0\}$ (see also Remark 5.8). In any case, our formulas are also presented  in an alternative  form involving only volume integrals, in which case extra regularity assumptions are not required.     
Although we do not enter in regularity issues for the solutions of Grushin-type equations,
 we note that the $W_0^{1,2}(\Omega)\cap W^{2,2}(\Omega)$ regularity assumption is satisfied for suitable classes of domains (for instance, if the domain is smooth and 
 has no characteristic points). In this regard, we refer to Kohn and Nirenberg \cite{KoNi65} and Jerison \cite{Je81}. 

Finally, we show two consequences of our analysis. First, motivated by shape optimization problems, 
 we characterize the critical shapes under isovolumetric and isoperimetric perturbations.  In Theorem \ref{criticcond},
 we prove that if a domain $\Omega$ is a critical set  under the volume constraint $\mathrm{Vol}(\Omega) = \mathrm{const.}$ for the symmetric functions of the eigenvalues bifurcating from an eigenvalue $\lambda$ of multiplicity $m$, then 
 \[
 \sum_{l = 1}^m \left(\frac{\partial v_l}{\partial \mathbf{n}}\right)^2|\mathbf{n}_G|^2 = c \qquad \mbox{ on } 
 \partial \Omega \setminus \{x=0\},
 \]
 for some constant $c$. 
Next, we consider the same problem under the perimeter  constraint  $\mathrm{Per}(\Omega) = \mathrm{const.}$ and 
in Theorem \ref{criticcondp} we  obtain the  additional condition
 \[
 \sum_{l = 1}^m \left(\frac{\partial v_l}{\partial \mathbf{n}}\right)^2|\mathbf{n}_G|^2 = c 
 \mathcal{H} \qquad \mbox{ on }   \partial \Omega \setminus \{x=0\},
 \]
 where $\mathcal{H}$ is the mean curvature of $\partial\Omega$. 
 Under suitable regularity assumptions on the eigenfunctions, 
  the above extra conditions  are also sufficient for $\Omega$ to be critical.
 As a second consequence, we obtain a new simple proof of the Rellich-Pohozaev identity for the Grushin 
 eigenvalues, {\it i.e.}
 \[
\lambda = \frac{1}{2}\int_{\partial\Omega}
  \left(\frac{\partial v}{\partial \mathbf{n}}\right)^2|\mathbf{n}_G|^2((x,(1+s)y) \cdot \mathbf{n})\,d\sigma_z,
\]
where $v$ is an eigenfunction  normalized in $L^2(\Omega)$.

The paper is organized as follows. In Section \ref{sec:not} we introduce some notation and preliminaries. Section 
  \ref{sec:tep} is devoted to the eigenvalue problem for the Dirichlet Grushin Laplacian and to 
  some well-known basic results about it. In Section  \ref{sec:dom} we define the set of admissible domain perturbations
  $\phi$'s and we prove that the $\phi$-pullback is a linear homeomorphism. Section \ref{sec:analit} contains our 
  main results, namely, we show that the symmetric functions of the eigenvalues depend real analytically upon shape 
  perturbations and we prove the Hadamard formula and the
  Rellich-Nagy-type theorem.  In Section \ref{sec:critical} 
  we characterize the critical shapes under isovolumetric and isoperimetric perturbations and 
  we formulate the corresponding overdetermined problems. 
  Finally, in Section \ref{sec:rellich} we provide a new proof of the  Relllich-Pohozaev 
  formula for the Grushin eigenvalues.

\section{Notation and preliminaries}\label{sec:not}
In order to deal with the Grushin Laplacian  $\Delta_{G}$, we need to introduce a well-known class
of associated weighted Sobolev spaces. Let $U$ be a bounded open subset of $\mathbb{R}^N$. 
We denote by 
$W_G^{1,2}(U)$  the space of real-valued functions in $L^2(U)$ such that 
$\partial_{x_i}u\in L^2(U)$  for all $i \in \{1,\ldots,h\}$ and 
 $ |x|^s \partial_{y_j}u\in L^2(U)$ for all  $j \in \{1,\ldots,k\}$.
The space $W_G^{1,2}(U)$ can be endowed with the following scalar product:
\begin{align*}
\langle u,v\rangle_{W_G^{1,2}(U)} : = \langle u,v\rangle_{L^2(U)} 
+ \sum_{i=1}^h\langle \partial_{x_i}u,\partial_{x_i} v\rangle_{L^2(U)}
 +\sum_{j=1}^k\langle  |x|^{s}\partial_{y_j}u,  |x|^s\partial_{y_j} v\rangle_{L^2(U)},
\end{align*}
for all $u,v \in W_G^{1,2}(U)$. Here $ \langle \cdot,\cdot\rangle_{L^2(U)}$ denotes the 
standard scalar product in $L^2(U)$. It is well known that the space $W_G^{1,2}(U)$ endowed with 
the scalar product $\langle \cdot,\cdot\rangle_{W_G^{1,2}(U)}$ is a Hilbert space.
The norm induced by the scalar product $\langle \cdot,\cdot\rangle_{W_G^{1,2}(U)}$ is 
\[
\|u\|_{W_G^{1,2}(U)} := \left(\|u\|_{L^2(U)}^2 +  \sum_{i=1}^h\|\partial_{x_i}u\|_{L^2(U)}^2
+  \sum_{j=1}^k\||x|^{2s}\partial_{y_j}u\|_{L^2(U)}^2\right)^{1/2} 
\]
for all  $u \in W_G^{1,2}(U)$.
Throughout the paper we use the following notation:
\begin{align*}
I_G(z) := 
 \begin{pmatrix} 
 I_{h \times h} & \mathbf{0}_{h \times k} \\
  \mathbf{0}_{k \times h}  & |x|^{s} I_{k \times k}   \\
 \end{pmatrix} \qquad \forall z=(x,y) \in \mathbb{R}^N,
 \end{align*}
 where $ I_{h \times h}$ and $ I_{k \times k}$ denote the $h \times h$ and $k \times k$ identity matrices, respectively, 
 whereas  $ \mathbf{0}_{h \times k}$ and $ \mathbf{0}_{k \times h}$ denote the  $h \times k$ and  $k \times h$ null matrices, respectively. Moreover, if $u \in W_G^{1,2}(U)$ we set
\begin{equation}\label{defnablaG}
\nabla_G u : = (\partial_{x_1}u,\ldots, \partial_{x_h}u, |x|^s\partial_{y_1}u, \ldots,  |x|^s\partial_{y_m}u) =  \nabla u \, I_G,
\end{equation}
and we refer to $\nabla_G u$ as the Grushin gradient of $u$.  We note that if $u$ is in $W_G^{1,2}(U)$, in general its gradient $\nabla u$ is a distribution. However, if by $\nabla u$ we mean the function defined a.e. in 
$U$ as the distributional gradient of $u$ in $U \setminus \{x=0\}$,  then the last equality of \eqref{defnablaG} is not only formal but holds almost everywhere.  The norm 
$\|\cdot\|_{W_G^{1,2}(U)}$ is equivalent to the norm $\|\cdot\|_{W_G^{1,2}(U)}'$ defined by
\[
\|u\|_{W_G^{1,2}(U)}' \equiv \|u\|_{L^2(U)} + \||\nabla_Gu|\|_{L^2(U)} \qquad \forall u \in W_G^{1,2}(U).
\]
\begin{remark}\label{eqnorms}
If $\overline U \cap \{x=0\} = \varnothing$,  then the norm in $W^{1,2}_G(U)$ is equivalent to the standard Sobolev 
norm of $W^{1,2}(U)$. 
\end{remark}
We denote by $W_{G,0}^{1,2}(U)$ the closure of $C_c^\infty(U)$ in $W_G^{1,2}(U)$.
In Theorem \ref{RK} below, we present an analog of the Rellich-Kondrachov embedding theorem that holds for the Sobolev 
space $W_{G,0}^{1,2}(U)$. For a proof we refer to the works of Franchi and Serapioni \cite[Theorem 4.6]{FrSe87} 
and of Kogoj and Lanconelli \cite[Proposition 3.2]{KoLa12}, which consider a more general class of weighted Sobolev 
spaces.
\begin{theorem}[Rellich-Kondrachov]\label{RK}
Let $U$ be a bounded open subset of $\mathbb{R}^N$.  Then the space $W_{G,0}^{1,2}(U)$ 
is compactly embedded in $L^2(U)$.
\end{theorem}

It is also known that an analog of the Poincar\'e inequality holds in the space  $W_{G,0}^{1,2}(U)$. Namely, the following theorem holds (for a proof, see, {\it e.g.}, D'Ambrosio \cite[Theorem 3.7]{Da04}, Monticelli and Payne \cite[Theorem 2.1]{MoPa09} and Monticelli, Payne and Punzo \cite{MoPaPu19}).  
\begin{theorem}[Poincar\'e inequality]\label{PI}
Let $U$ be a bounded open subset of $\mathbb{R}^N$.  Then there exists $C>0$ such that
\begin{equation*}
\|u\|_{L^2(U)} \leq C \||\nabla_G u|\|_{L^2(U)} \qquad \forall u \in W_{G,0}^{1,2}(U).
\end{equation*}
\end{theorem}

\section{The eigenvalue problem}\label{sec:tep}
Here we introduce the precise formulation of the eigenvalue problem. We fix  $U$ to be a bounded open subset of $\mathbb{R}^N$.  The classical spectral problem is  
\begin{equation}\label{ssp}
\begin{cases}
-\Delta_{G} u = \lambda  u \qquad &\mbox{ in } U,\\
u=0 \qquad &\mbox{ on } \partial U,
\end{cases}
\end{equation}
in the unknowns $\lambda$ (the eigenvalues) and $u$ (the eigenfunctions). Actually, in order to reduce 
the regularity assumptions, we consider the weak formulation of  problem \eqref{ssp}. Namely, 
\begin{equation}\label{wsp}
\int_{U}\nabla_G u \cdot\nabla_G\varphi\,dz = \lambda \int_{U} u\varphi\,dz \qquad \forall 
\varphi \in W_{G,0}^{1,2}(U),
\end{equation}
in the unknowns $\lambda \in \mathbb{R}$ and $u \in W_{G,0}^{1,2}(U)$. We use a standard procedure which enables us to reduce the study of the
eigenvalues of \eqref{wsp} to an eigenvalue problem for a compact self-adjoint operator in a Hilbert space.
With a slight abuse of notation, we consider the Grushin Laplacian $\Delta_{G}$ as the operator from $W_{G,0}^{1,2}(U)$ to 
its dual $(W_{G,0}^{1,2}(U))'$ defined by  
\begin{equation}\label{DGdef}
\Delta_{G}[u][\varphi] := - \int_{U}\nabla_G u \cdot\nabla_G\varphi\,dz \qquad \forall u,\varphi \in W_{G,0}^{1,2}(U).
\end{equation}
Next, we define 
the following bilinear form 
\[
Q_G[u,v] := - \Delta_{G}[u][v]  \qquad  \forall u,v \in W_{G,0}^{1,2}(U).
\]
It is easily seen that the bilinear form $Q_G$ is continuous. Moreover, by the Poincar\'e inequality  of 
Theorem \ref{PI}, we have that
\[
Q_G[u,u] = \int_{U}|\nabla_G u|^2\,dz \geq c_2 \|u\|_{W^{1,2}_G(U)}^2
 \qquad  \forall u \in W_{G,0}^{1,2}(U),
\]
for some $c_2>0$ and thus the bilinear form $Q_G$ is coercive. In other words, $Q_G$ 
is a scalar product on $W_{G,0}^{1,2}(U)$ which induces a norm equivalent to the standard one. 
Thus, we can apply the Riesz representation theorem 
to deduce that  $\Delta_{G}$ is a linear  homeomorphism from $W_{G,0}^{1,2}(U)$ onto 
$(W_{G,0}^{1,2}(U))'$. 
We denote by $J$ the map from $L^2(U)$ to $(W_{G,0}^{1,2}(U))'$ defined by
\begin{equation}\label{jdef}
J[u][\psi] := \langle u, \psi \rangle_{L^2(U)}= \int_{U} u\psi\,dz \qquad \forall u \in L^2(U), \, 
\forall\psi \in W_{G,0}^{1,2}(U).
\end{equation}
Clearly $J$ is continuous and injective. Equation \eqref{wsp} can be rewritten as 
\begin{equation}\label{wspr}
- \Delta_{G}^{(-1)} \circ J \circ i[u] = \mu u,
\end{equation}
where $\mu = \lambda^{-1}$ and $i$ is the embedding of $ W_{G,0}^{1,2}(U)$ in $L^2(U)$.
Accordingly, it is natural to consider the operator $T_G$ from  $W_{G,0}^{1,2}(U)$ to itself 
defined by 
\begin{equation*}
T_G[u] := - \Delta_{G}^{(-1)} \circ J \circ i[u] \qquad \forall u \in W_{G,0}^{1,2}(U).
\end{equation*}
 Since the embedding $i$ is 
compact by  Theorem \ref{RK}, $T_G$ is compact in  $W_{G,0}^{1,2}(U)$. 
Moreover, $T_G$ is self-adjoint in $W_{G,0}^{1,2}(U)$ endowed with the scalar product $Q_G$. Indeed,
\begin{align*}
Q_G[T_Gu,v] =-\Delta_{G}[T_Gu][v]=\Delta_{G}\left[\Delta_{G}^{(-1)}\circ J\circ i[u]\right][v]= \int_{U}uv\,dz\,
\qquad \forall u,v \in W_{G,0}^{1,2}(U).
\end{align*}
Since $Q_G$ is symmetric, we have that $Q_G[T_Gu,v]= Q_G[u,T_Gv]$. 
In addition,  $T_G$ is injective because it is the composition of injective maps. 
It follows that the spectrum of   $T_G$ is discrete and consists of a sequence of positive eigenvalues
$\mu_j[U]$ of finite multiplicity converging to zero. More precisely, by classical spectral theory, by the min-max principle (see, {\it e.g.}, Davies \cite[\S 4.5]{Da95}), and by the equivalence of the formulations \eqref{wsp} and \eqref{wspr} 
 we have the following.

\begin{theorem}
The eigenvalues of equation \eqref{wsp} have finite multiplicity and can be represented by means of a divergent sequence 
\[ 
0 < \lambda_1[U] \leq \lambda_2[U] \leq \ldots \leq\lambda_j[U]\leq \ldots \to +\infty.
\]
Moreover, they coincide with the inverse of the eigenvalues $\mu_j[U]$ of $T_G$, and
\begin{equation*}\label{min-max}
\lambda_j[U] = \min_{\substack{E \subseteq W_{G,0}^{1,2}(U)\\ \mathrm{dim} E = j}} 
\max_{\substack{u \in E \\ u\neq 0}} \frac{\int_{U}|\nabla_G u|^2\,dz}{\int_{U} u^2 \, dz} 
\qquad \forall j \in \mathbb{N}.
\end{equation*}
\end{theorem}

\section{Admissible domain perturbations}\label{sec:dom}
Since we plan to consider the eigenvalue problem \eqref{wsp} on a variable domain, the first step is to define 
what we mean by variable domain. Our point of view is to consider a fixed domain and  a family of open sets parametrized by a suitable homeomorphism defined on the fixed domain.  Accordingly, we fix 
\begin{equation}\label{Omega_def}
\begin{split}
&\text{a bounded open subset $\Omega$ of $\mathbb{R}^N$ and a bounded open subset  $O$ of $\mathbb{R}^N$}
\\
&\text{such that $\varnothing\neq  \overline \Omega \cap \{x=0\} \subseteq O$.}
\end{split}
\end{equation}
From now on if $\phi$ is a map with values in $\mathbb{R}^N$, we denote by $\phi_x = (\phi_{x_1}, \ldots, \phi_{x_h})$ and by $\phi_y= (\phi_{y_1}, \ldots, \phi_{y_k})$ the 
 first $h$ and the last $k$ components of $\phi$, respectively. Moreover, we denote by $\pi_x$ and by $\pi_y$  the projections of $\mathbb{R}^N$ to $\mathbb{R}^h$ and  $\mathbb{R}^k$ which  take 
$z=(x,y)$ to $x$ and $y$, respectively.  We set 
\begin{align*}
 L_{\Omega,O} := \Big\{ \phi   \in   \mathrm{Lip}(\Omega)^N:\,\, 
 &\exists\, \tilde \phi_x \in\mathrm{Lip}(\pi_x(\Omega \cap O))^h ,\,
 \tilde \phi_y \in\mathrm{Lip}(\pi_y(\Omega \cap O))^k \\
 &\mbox{ s.t. } \phi =( \tilde \phi_x \circ \pi_x ,  \tilde \phi_y \circ \pi_y) \, \mbox{ in } \Omega \cap O, \,\,\tilde \phi_x(0) =0 \Big\}.
 \end{align*}
It is 
 easily seen that ${L}_{\Omega, O}$ is a closed linear subspace of the Banach space $\mathrm{Lip}(\Omega)^N$,  where  $\mathrm{Lip}(\Omega)$ denotes the Banach space of Lipschitz functions in $\Omega$ endowed with the norm $\sup_\Omega|f| +\sup_{\substack{z_1,z_2 \in \Omega \\ z_1 \neq z_2}} \frac{|f(z_1)-f(z_2)|}{|z_1-z_2|}$.  Therefore, ${L}_{\Omega, O}$ is a Banach space itself. 
 We define the space of admissible shape perturbations as 
\begin{align*}
\mathcal{A}_{\Omega, O} := \bigg\{ \phi    \in    L_{\Omega,O}: &
 \inf_{\substack{ z_1,z_2 \in  \Omega\\ z_1 \neq z_2} }
\frac{|\phi(z_1)-\phi(z_2)|}{|z_1-z_2|}>0\,\bigg\}.
\end{align*}
 By Lamberti and Lanza de Cristoforis \cite[Lemma 3.11]{LaLa04}, if $\phi \in \mathcal{A}_{\Omega,O}$ then $\phi$ is injective  and 
 $\inf_{\Omega}|\det D\phi| >0$. Moreover, $\phi(\Omega)$ is a bounded open set and the inverse map $\phi^{(-1)}$ belongs to $\mathcal{A}_{\phi(\Omega), O'}$ for some open set $O'$ containing 
 $\overline{\phi(\Omega)} \cap \{x=0\}$.
 \begin{remark}
 If  $\phi \in \mathcal{A}_{\Omega,O}$, then $\phi$ is a bi-Lipschitz homeomorphism from 
 $\Omega$ into his image that near the degenerate set, {\it i.e.} inside $\Omega \cap O$,  deforms separately the $x$-direction  and the $y$-direction. Moreover, if a point belongs to the degenerate set $\{x=0\}$, then its image through $\phi$ has to remain on the degenerate set. Since $\phi$ is bi-Lipschitz, it is easily seen that there exists $C>0$ such that $\frac{1}{C}|x|\leq \phi_x(z) \leq C|x|$ for all $z \in \Omega$. 
 Finally, it is worth noting that our setting includes both the case in which the degenerate set 
 $\{x=0\}$  intersects $\Omega$ and the case in which part of the boundary of $\Omega$ lies on the degenerate set.
 \end{remark}
 \begin{figure}[htb]
 \captionsetup{justification=centering}

\centering
\includegraphics[width=6in]{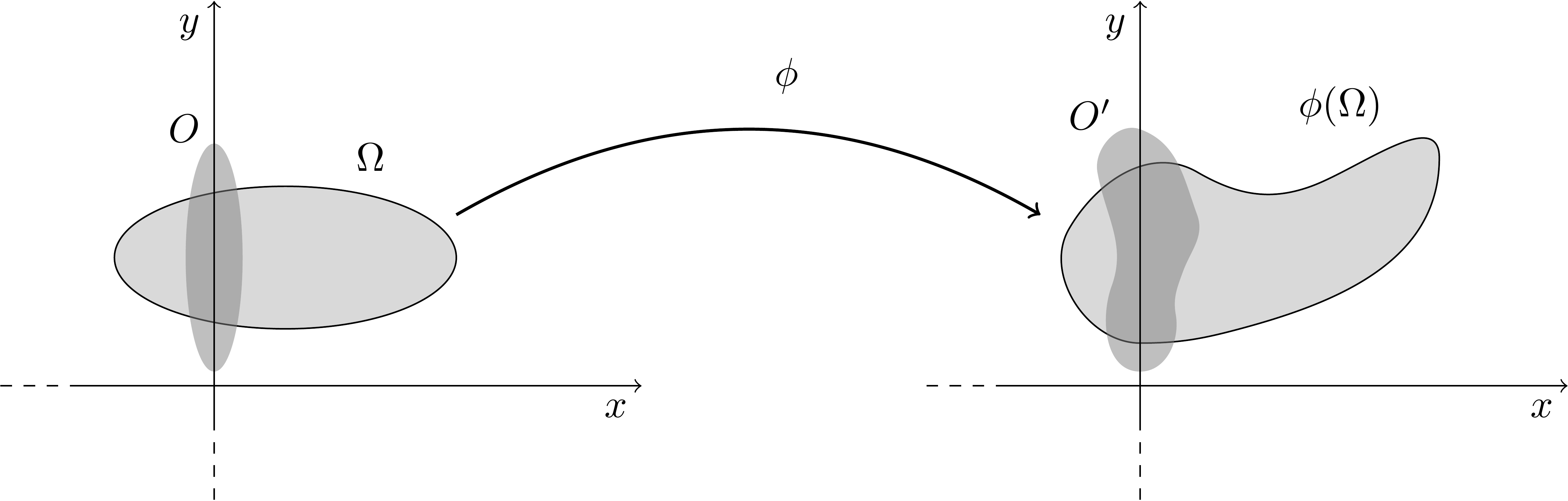}
\begin{center}
\caption{{\it An example, in dimension $N=2$, of a homeomorphism $\phi \in \mathcal{A}_{\Omega,O}$ 
 for a  possible choice of fixed $\Omega$ and $O$.  In the intersection between $\Omega$ and $O$, the 
 homeomorphism $\phi$  deforms separately the $x$ and $y$ directions.}}\label{fig:DiffeoOm}
\end{center}
\end{figure} 
 For a  transformation $\phi \in \mathcal{A}_{\Omega,O}$, we are able to prove that the $\phi$-pushforward (or equivalently the $\phi$-pullback), that we will use to transplant the problem to the fixed domain $\Omega$, is a linear homeomorphism.
 \begin{lemma}\label{comphomeo}
 Let $\Omega$ and $O$ be as in \eqref{Omega_def}. Let  $\phi \in \mathcal{A}_{\Omega,O}$. 
 Then the operator $C_{\phi^{(-1)}}$ defined by 
 \[
 C_{\phi^{(-1)}}[u] := u \circ \phi^{(-1)} \qquad \forall u \in L^2(\Omega),
 \]
 is a linear homeomorphism  from  $L^2(\Omega)$ to $L^2(\phi(\Omega))$ which restricts  a linear homeomorphism from $W^{1,2}_{G}(\Omega)$ onto $W^{1,2}_{G}(\phi(\Omega))$ and from $W^{1,2}_{G,0}(\Omega)$ onto $W^{1,2}_{G,0}(\phi(\Omega))$. Moreover 
 $C_{\phi^{(-1)}}^{(-1)} = C_{\phi}$. 
 \end{lemma}
 \begin{proof}
 Let $u \in  L^2(\Omega)$. There exists  $c_1>0$ such that 
 \begin{align*}
 \| C_{\phi^{(-1)}}[u]\|_{L^2(\phi(\Omega))}^2 &=  \int_{\phi(\Omega)}(u \circ \phi^{(-1)}(z))^2 \, dz  
 = \int_{\Omega}(u(z))^2 |\det D\phi|\, dz \\
 & \leq 
  c_1  \int_{\Omega}(u(z))^2\, dz= c_1\|u\|_{L^2(\Omega)}^2 .
 \end{align*}
 Thus, $C_{\phi^{(-1)}}$ is continuous from $L^2(\Omega)$ to 
$L^2(\phi(\Omega))$. Since  $C_{\phi^{(-1)}}$ is clearly surjective, the Open Mapping Theorem implies that it  is a linear homeomorphism from $L^2(\Omega)$ to 
$L^2(\phi(\Omega))$. 

Next we fix $u \in C^1(\Omega) \cap W^{1,2}_{G}(\Omega)$.
Since $\phi$ is invertible,  we have that $\phi(\Omega \cap O) \cap \phi(\Omega \setminus O)= \varnothing$ and  that $\phi(\Omega \cap O) \cup \phi(\Omega \setminus O)= \phi(\Omega)$ and thus
 \begin{equation}\label{eq:comphomeo:1}
 \begin{split}
 \| |\nabla_GC_{\phi^{(-1)}}[u]|\|_{L^2(\phi(\Omega))}^2&= \int_{\phi(\Omega)}|\nabla_G(u \circ \phi^{(-1)})(z)|^2 \, dz \\
 =&  \int_{\phi(\Omega\cap O)}|\nabla_G(u \circ \phi^{(-1)})(z)|^2 \, dz \\
 &+ \int_{\phi(\Omega \setminus O)}|\nabla_G(u \circ \phi^{(-1)})(z)|^2 \, dz.
   \end{split}
   \end{equation}
We now consider the first summand $\int_{\phi(\Omega\cap O)}|\nabla_G(u \circ \phi^{(-1)})(z)|^2 \, dz$ in the right hand side of \eqref{eq:comphomeo:1}. We have
  \begin{align*}
    \int_{\phi(\Omega\cap O)}|\nabla_G(u \circ \phi^{(-1)})(z)|^2 \, dz 
    =& \int_{\phi(\Omega\cap O)}|\nabla u(\phi^{(-1)}(z))(D\phi^{(-1)}(z))I_G(z)|^2 \, dz \\
    =&\int_{\Omega \cap O}|\nabla u(z)(D\phi(z))^{-1}I_G(\phi(z))|^2 |\det D\phi(z)|\, dz .
  \end{align*}
  We now observe that for a.a. $z \in \Omega \cap O$ we have 
    \begin{align*}
  (D\phi(z))^{-1}I_G(\phi(z)) =  &\begin{pmatrix} 
D_x\phi_x(z) & \mathbf{0}_{h \times k} \\
  \mathbf{0}_{k \times h}  & D_y\phi_y(z)  \\
 \end{pmatrix}^{-1}
  \begin{pmatrix} 
 I_{h \times h} & \mathbf{0}_{h \times k} \\
  \mathbf{0}_{k \times h}  & |\phi_x(z)|^{s} I_{k \times k}   \\
 \end{pmatrix}\\
 &=\begin{pmatrix} 
 (D_x\phi_x(z))^{-1} & \mathbf{0}_{h \times k} \\
  \mathbf{0}_{k \times h}  & (D_y\phi_y(z))^{-1}  \\
 \end{pmatrix}
 \begin{pmatrix} 
 I_{h \times h} & \mathbf{0}_{h \times k} \\
  \mathbf{0}_{k \times h}  & |\phi_x(z)|^{s} I_{k \times k}   \\
 \end{pmatrix}\\
 &= \begin{pmatrix} 
 I_{h \times h} & \mathbf{0}_{h \times k} \\
  \mathbf{0}_{k \times h}  & |\phi_x(z)|^{s} I_{k \times k}   \\
 \end{pmatrix}
 \begin{pmatrix} 
 (D_x\phi_x(z))^{-1} & \mathbf{0}_{h \times k} \\
  \mathbf{0}_{k \times h}  & (D_y\phi_y(z))^{-1} \\
 \end{pmatrix}\\
 &= I_G(\phi(z)) (D\phi(z))^{-1} \, .
  \end{align*}
Thus,
\[
\begin{split}
\int_{\Omega \cap O}|\nabla u(z)(D\phi(z))^{-1}I_G(\phi(z))|^2& |\det D\phi(z)|\, dz\\&=  \int_{\Omega \cap O}|\nabla u(z)I_G(\phi(z))(D\phi(z))^{-1}|^2|\det D\phi(z)| \, dz\, .
\end{split}
\]
Then we note that
\[
\nabla u(z)I_G(\phi(z))=(\nabla_xu(z),|\phi_x(z)|^s\nabla_yu(z)) \qquad \mbox{ for a.a. } z \in \Omega, 
\] 
 and  that there exists a constant $C>0$ such that $\frac{1}{C}|x|\leq|\phi_x(z)| \leq C |x|$ for all $z \in \Omega$. As a consequence, since $\phi \in \mathrm{Lip}(\Omega)^N$, we deduce the existence of $c_2>0$ such that 
\[
 \int_{\Omega \cap O}|\nabla u(z)I_G(\phi(z))(D\phi(z))^{-1}|^2|\det D\phi(z)| \, dz\leq c_2 \int_{\Omega \cap O}|\nabla_G u(z)|^2 \, dz\, ,
\]
and accordingly that 
\begin{equation}\label{eq:comphomeo:2}
 \int_{\phi(\Omega\cap O)}|\nabla_G(u \circ \phi^{(-1)})(z)|^2 \, dz\leq c_2 \int_{\Omega \cap O}|\nabla_G u(z)|^2 \, dz\, .
\end{equation}
We now turn to the second summand $\int_{\phi(\Omega \setminus O)}|\nabla_G(u \circ \phi^{(-1)})(z)|^2 \, dz$ in the right hand side of \eqref{eq:comphomeo:1}.    By Remark \ref{eqnorms},  the $W^{1,2}_G$-norm  is equivalent to the standard Sobolev 
norm of $W^{1,2}$ if we are far from the degenerate set $\{x=0\}$ and thus there exist $c_3,c_4,c_5 >0$ such that
  \begin{equation}\label{eq:comphomeo:3}
  \begin{split}
 \int_{\phi(\Omega \setminus O)}|\nabla_G(u \circ \phi^{(-1)})(z)|^2 \, dz 
  \leq&\,c_3\int_{\phi(\Omega \setminus O)}|\nabla(u \circ \phi^{(-1)})(z)|^2 \, dz\\
   =&\,c_3\int_{\Omega \setminus O}|\nabla u(z)(D\phi(z))^{-1}|^2 |\det D\phi(z)|\, dz\\
   \leq &\,c_4\int_{\Omega \setminus O}|\nabla u(z)|^2 \, dz\\
   \leq &\,c_5\int_{\Omega \setminus O}|\nabla_G u(z)|^2 \, dz\, .
  \end{split}
  \end{equation}
 Thus by \eqref{eq:comphomeo:1} and by summing up the inequalities in \eqref{eq:comphomeo:2} and \eqref{eq:comphomeo:3}, there exists $c_6>0$ such that
    \begin{align*}
 \| |\nabla_GC_{\phi^{(-1)}}[u]|\|_{L^2(\phi(\Omega))}^2
  \leq c_6   \| |\nabla_Gu|\|_{L^2(\Omega)}^2.
  \end{align*}
Since $C_{\phi^{(-1)}}$ is continuous from $L^2(\Omega)$ to  $L^2(\phi(\Omega))$, since $W^{1,2}_{G}(\Omega)$  is continuously embedded in $L^2(\Omega)$, and since $C^1(\Omega) \cap W^{1,2}_{G}(\Omega)$ is dense in  $W^{1,2}_{G}(\Omega)$ (see Franchi, Serapioni and Serra Cassano \cite{FrSeSe96}), one can realize that  $C_{\phi^{(-1)}}$ is continuous from $W^{1,2}_G(\Omega)$ to 
$W^{1,2}_G(\phi(\Omega))$. To show the surjectivity, we take $v \in   W^{1,2}_G(\phi(\Omega))$. Following the same argument  as above together with the inequality $\frac{1}{C}|x| \leq |\phi_x(z)|$  
  one can realize that $v \circ \phi \in W^{1,2}_{G}(\Omega)$ and, clearly, $C_{\phi^{(-1)}}[v \circ \phi]=v$. By the Open Mapping Theorem  $C_{\phi^{(-1)}}$ is a linear homeomorphism from $W^{1,2}_G(\Omega)$ to 
$W^{1,2}_G(\phi(\Omega))$.  

Finally, by a standard mollification argument  $C_{\phi^{(-1)}}[u]  \in W^{1,2}_{G,0}(\phi(\Omega))$ for all $u \in C^\infty_c(\phi(\Omega))$. Therefore,
 since $W^{1,2}_{G,0}(\phi(\Omega))$ is a closed subspace of $W^{1,2}_G(\phi(\Omega))$, we have that 
$C_{\phi^{(-1)}}[u]  \in W^{1,2}_{G,0}(\phi(\Omega))$ for all $u \in W^{1,2}_{G,0}(\Omega)$ and thus 
$C_{\phi^{(-1)}}$ restricts a linear homeomorphism from $W^{1,2}_{G,0}(\Omega)$ onto
 $W^{1,2}_{G,0}(\phi(\Omega))$.
  The last part of the statement is obvious. 
 \end{proof}

 \section{Analyticity results and Hadamard formula}\label{sec:analit}
In this section we perform the shape sensitivity analysis of the Grushin eigenvalue problem.  As in the previous section,  we fix 
$\Omega$ and 
$O$ as in \eqref{Omega_def} and  $\phi \in \mathcal{A}_{\Omega,O}$. We consider
\begin{equation}\label{phiwsp}
\int_{\phi(\Omega)}\nabla_G v \cdot\nabla_G\psi\,dz = \lambda \int_{\phi(\Omega)} v\psi\,dz \qquad \forall 
\psi \in W_{G,0}^{1,2}(\phi(\Omega)),
\end{equation} 
in the unknowns $v \in W^{1,2}_{G,0}(\phi(\Omega))$ and $\lambda \in \mathbb{R}$.
By the results of  Section \ref{sec:tep}, the eigenvalues of equation \eqref{phiwsp} have finite multiplicity and can be represented by means of a divergent sequence 
\[
0 < \lambda_1[\phi] \leq \lambda_2[\phi] \leq \ldots \leq\lambda_j[\phi]\leq \ldots \to +\infty,
\]
where we have set 
\[
 \lambda_j[\phi] :=  \lambda_j[\phi(\Omega)] \qquad \forall j \in \mathbb{N}.
\]
In general, if we want to study the regularity of an eigenvalue upon a parameter, 
which in our case  is  $\phi$, we face a first 
problem. Namely, we cannot expect to prove smooth dependence of the eigenvalues themselves upon the parameter, when the eigenvalues are not simple. This is due to bifurcation phenomena of splitting from a multiple eigenvalue to different eigenvalues of lower multiplicity (cf.~Rellich \cite[p.~37]{Re69}). Hence, to circumvent this problem, we consider the 
 elementary symmetric functions of the eigenvalues. 
 This is the point of view introduced by Lamberti and Lanza de Cristoforis
 \cite{LaLa04} and later adopted in many other works (see, {\it e.g.}, \cite{BuLa13, BuLa15, La09, LaPr13}). 
Clearly, when a certain eigenvalue is simple, for example in the case of the first Grushin eigenvalue under the 
 assumption that $\Omega \setminus \{x=0\}$ is connected (see Monticelli and Payne \cite[Theorem 6.4]{MoPa09}), 
 our regularity result for the symmetric functions of the eigenvalues implies that the same regularity is valid for the eigenvalue.

 To perform this strategy, we need to introduce two subspaces of $\mathcal{A}_{\Omega,O}$.  Let 
 $F \subseteq \mathbb{N}$ be a finite set of indexes and we consider the subset of 
 $\mathcal{A}_{\Omega,O}$ of those maps  for which the eigenvalues with index in $F$ do not coincide with the eigenvalues with index outside $F$. That is
 \begin{align*}
 \mathcal{A}^F_{\Omega,O} := \big\{\phi \in  \mathcal{A}_{\Omega,O} :  \lambda_n[\phi] \neq \lambda_m[\phi], \,
 \forall n \in F, \,\forall m \in \mathbb{N} \setminus F\big\}.
 \end{align*}
 We find also convenient to consider the set  $\Theta^F_{\Omega,O}$ of those maps in  
 $\mathcal{A}^F_{\Omega,O}$
 such that all the eigenvalues with index in $F$ coincide. Namely, 
 \begin{align*}
 \Theta^F_{\Omega,O} := \big\{\phi \in  \mathcal{A}^F_{\Omega,O} :  \lambda_n[\phi] = \lambda_m[\phi], \,
 \forall n,m \in F\big\}.
 \end{align*}
 For $\phi \in \mathcal{A}_{\Omega,O}$ we introduce the following two operators. 
 \begin{itemize}
 \item[i)] $J_\phi$ is the map from $L^2(\Omega)$ to $(W^{1,2}_{G,0}(\Omega))'$ defined by
 \[
 J_\phi[u][v] := \int_{\Omega}uv |\det D\phi| \,dz \qquad \forall u \in L^2(\Omega),\, \forall v \in W^{1,2}_{G,0}(\Omega).
 \]
  \item[ii)]  $\Delta_{G,\phi}$ is the map from $W^{1,2}_{G,0}(\Omega)$ to $(W^{1,2}_{G,0}(\Omega))'$  defined by
 \[
\Delta_{G,\phi}[u][v] := -\int_{\Omega}\nabla u  (D\phi)^{-1} I_G(\phi) (\nabla v(D\phi)^{-1} I_G(\phi))^t|\det D\phi| \,dz \quad \forall u,v \in W^{1,2}_{G,0}(\Omega).
 \]
 \end{itemize}
\begin{remark}
Let $\phi \in \mathcal{A}_{\Omega,O}$. By performing a change of variable, one can readily verify that
\[
J_\phi = C_{\phi^{(-1)}}^t \circ J \circ C_{\phi^{(-1)}} \qquad 
\Delta_{G,\phi}= C_{\phi^{(-1)}}^t \circ \Delta_{G} \circ C_{\phi^{(-1)}},
\]
being $J$ and $\Delta_{G}$ the operators defined in \eqref{jdef} and \eqref{DGdef} with $U= \phi(\Omega)$, 
respectively, and $C_{\phi^{(-1)}}$ the $\phi$-pushforward operator introduced in Lemma \ref{comphomeo}.
\end{remark}
Since $\Delta_{G}$ is a linear homeomorphism from $W^{1,2}_{G,0}(\phi(\Omega))$ onto 
$(W^{1,2}_{G,0}(\phi(\Omega)))'$ and since $J$ is a linear and continuous injection 
from $L^2(\phi(\Omega))$ to $(W^{1,2}_{G,0}(\phi(\Omega)))'$, Lemma \ref{comphomeo} immediately implies the following.
\begin{corollary}
 Let $\Omega$
and $O$ be as in \eqref{Omega_def} and  $\phi \in \mathcal{A}_{\Omega,O}$.
 Then the operator $\Delta_{G,\phi}$ is a linear homeomorphism 
 $W^{1,2}_{G,0}(\Omega)$ onto 
$(W^{1,2}_{G,0}(\Omega))'$ and the operator $J_\phi$ is a linear and 
continuous injection 
from $L^2(\Omega)$ to $(W^{1,2}_{G,0}(\Omega))'$.
\end{corollary}
Next, in order to reformulate problem \eqref{phiwsp} into a spectral problem for a compact self-adjoint operator, we set $T_{G,\phi}$  to be the map from $W^{1,2}_{G,0}(\Omega)$ to itself defined by
 \begin{equation}\label{def:Tphi}
 T_{G,\phi}[u] := - \Delta_{G,\phi}^{(-1)} \circ J_\phi \circ i[u] \qquad \forall u \in W^{1,2}_{G,0}(\Omega).
 \end{equation}
Here, $i$ denotes the embedding of $W^{1,2}_{G,0}(\Omega)$ in $L^2(\Omega)$.
 Clearly, equation \eqref{phiwsp}  is equivalent to 
 \[
 T_{G,\phi}[u] = \mu u
 \]
 with $u = v \circ \phi$ and $\mu = \lambda^{-1}$. Furthermore, we set
 \begin{equation}\label{Qdefphi}
 Q_{G,\phi}[u,v] := -\Delta_{G,\phi}[u][v] \qquad \forall u,v \in   W^{1,2}_{G,0}(\Omega).
 \end{equation}
 Adapting the same computations of the proof of Lemma \ref{comphomeo}, it is easily seen that $Q_{G,\phi}$ is a scalar product on $W^{1,2}_{G,0}(\Omega)$ 
 which induces a norm equivalent to the standard one in $W^{1,2}_{G,0}(\Omega)$. 
 
We now consider the operator $T_{G,\phi}$  acting on  $\left(W^{1,2}_{G,0}(\Omega),Q_{G,\phi}\right)$ and we prove that it is a compact self-adjoint operator and that it depends real analytically on $\phi$.
Before doing this, we need nome notation.
If $\mathcal{X}$, $\mathcal{Y}$ are two Banach spaces, we denote by $\mathcal{L}(\mathcal{X},\mathcal{Y})$ the space of linear and continuous operators 
  from $\mathcal{X}$ to $\mathcal{Y}$, we set $\mathcal{L}(\mathcal{X}) := \mathcal{L}(\mathcal{X},\mathcal{X})$ and we denote by $\mathcal{B}_s(\mathcal{X})$ the space 
  of bilinear symmetric forms on $\mathcal{X}$.  These spaces are endowed with their standard norms.
 \begin{proposition}\label{propTp}
  Let  $\Omega$ and
$O$ be as in \eqref{Omega_def} and  $\phi \in \mathcal{A}_{\Omega,O}$.
  Then 
  \begin{itemize}
  \item[(i)] $T_{G,\phi}$ is a compact self-adjoint operator in 
  $\left(W^{1,2}_{G,0}(\Omega),Q_{G,\phi}\right)$.
  \item[(ii)] The map from $\mathcal{A}_{\Omega,O}$ to 
  $\mathcal{L}\left(W^{1,2}_{G,0}(\Omega)\right) \times \mathcal{B}_s\left(W^{1,2}_{G,0}(\Omega)\right)
  $ which takes $\phi$ to $(T_{G,\phi}, Q_{G,\phi})$ is real analytic.
  \end{itemize}
 \end{proposition}
 \begin{proof}
 First we consider statement (i).   The compactness of $T_{G,\phi}$ is a
  consequence of the compactness of the embedding of $W^{1,2}_{G,0}(\Omega)$ in  $L^2(\Omega)$. For the 
  self-adjointess we note that
  \begin{align*}
  Q_{G,\phi}[T_{G,\phi} u,v] = -\Delta_{G,\phi}[T_{G,\phi} u][v] =&
  -\Delta_{G,\phi}\left[-\Delta_{G,\phi}^{(-1)} \circ J_\phi \circ i[u]\right][v] \\
  =& \,J_\phi[i[u]][v] \qquad\qquad \forall u,v \in W^{1,2}_{G,0}(\Omega).
  \end{align*}
  
Next, we prove statement (ii). It is easily seen that the maps which take $\phi$ to $\Delta_{G,\phi}$, 
$J_\phi$ and $Q_{G,\phi}$ from $\mathcal{A}_{\Omega,O}$ to 
$\mathcal{L}(W^{1,2}_{G,0}(\Omega), (W^{1,2}_{G,0}(\Omega))')$, 
$\mathcal{L}(L^2(\Omega), (W^{1,2}_{G,0}(\Omega))')$ and $\mathcal{B}_s(W^{1,2}_{G,0}(\Omega))$, respectively, are real analytic. Then, since the map which takes an invertible operator to its inverse 
is real analytic we can conclude that $T_{G,\phi}$ depends real analytically on $\phi$.
 \end{proof}
We are ready to prove that the elementary symmetric functions of the eigenvalues depends real analytically upon the domain's shape $\phi$. 
\begin{theorem}\label{thm:realanalsym}
 Let  $\Omega$ and 
$O$ be as in \eqref{Omega_def}.  Let $F$ be a finite nonempty subset 
 of $\mathbb{N}$. 
 Let $\tau \in \{1,\ldots,|F|\}$. Then  $\mathcal{A}^F_{\Omega,O}$ is open in  
 $L_{\Omega,O}$ and the map $\Lambda_{F,\tau}$ from $\mathcal{A}^F_{\Omega,O}$ to $\mathbb{R}$   
defined by
\[
\Lambda_{F,\tau}[\phi] := \sum_{\substack{j_1,\ldots,j_\tau \in F \\ j_1<\cdots<j_\tau}}\lambda_{j_1}[\phi] \cdots \lambda_{j_\tau}[\phi] \qquad \forall \phi \in \mathcal{A}^F_{\Omega,O}
\]
is real analytic.
\end{theorem}
 \begin{proof}
We denote by $\{\mu_j[\phi]\}_{j \in \mathbb{N}\setminus\{0\}}$ the set of eigenvalues of $T_{G,\phi}$. As
 we have already pointed out $\mu_j[\phi] = (\lambda_j[\phi])^{-1}$. Hence, the set $\mathcal{A}^F_{\Omega,O}$
 coincides with the set 
 \[
 \big\{\phi \in \mathcal{A}_{\Omega,O}  :  \mu_n[\phi] \neq \mu_m[\phi], \,
 \forall n \in F, \,\forall m \in \mathbb{N} \setminus F \big\}.
 \]
 By Proposition \ref{propTp}, $T_{G,\phi}$ is self-adjoint with respect to the scalar product $Q_{G,\phi}$ 
 and, moreover, both  $T_{G,\phi}$ and $Q_{G,\phi}$  depend analytically on 
 $\phi \in \mathcal{A}^F_{\Omega,O}$. Thus, Lamberti and Lanza de Cristoforis  \cite[Theorem 2.30]{LaLa04} implies 
 that $\mathcal{A}^F_{\Omega,O}$ is open in  $L_{\Omega,O}$ and the map 
 $M_{F,\tau}$ from $\mathcal{A}^F_{\Omega,O}$ to $\mathbb{R}$ defined by
 \begin{equation}\label{Mdef}
 M_{F,\tau}[\phi] :=  \sum_{\substack{j_1,\ldots,j_\tau \in F \\ j_1<\cdots<j_\tau}}\mu_{j_1}[\phi] \cdots \mu_{j_\tau}[\phi] \qquad \forall \phi \in \mathcal{A}^F_{\Omega,O}
 \end{equation}
 is real analytic. If we set $M_{F,0}[\phi] := 1$ for all $\phi \in \mathcal{A}^F_{\Omega,O}$, one can 
 readily verify that 
 \begin{equation}\label{LM}
 \Lambda_{F,\tau}[\phi]  = \frac{M_{F,|F|-\tau}[\phi]}{M_{F,|F|}[\phi]} 
 \qquad \forall \phi \in \mathcal{A}^F_{\Omega,O}.
 \end{equation}
 Accordingly the statement follows.
 \end{proof}
 In view of the applications,  once we have considered the regularity of the elementary symmetric functions, it is important to have an explicit formula for their shape differential. Thus, our next step is to prove an Hadamard-type formula for the shape differential of the elementary symmetric functions.
 \begin{theorem}\label{hadf}
Let  $\Omega$ and  $O$ be as in \eqref{Omega_def}.
   Let $F$ be a finite nonempty subset 
 of $\mathbb{N}$. Let $\tau \in \{1,\ldots,|F|\}$.  
Let $\tilde \phi \in \Theta^F_{\Omega,O}$ and  let
 $\lambda_F[\tilde \phi]$ be the common value of all the eigenvalues $\{\lambda_j[\tilde \phi]\}_{j \in F}$. 
 Let $\{v_l\}_{l \in F}$ be an orthonormal basis in 
 $(W^{1,2}_{G,0}(\tilde \phi(\Omega)),Q_{G})$
  of the eigenspace associated with $\lambda_F[\tilde \phi]$.
Then  the Frech\'et differential of the map  $\Lambda_{F,\tau}$ at the
 point $\tilde \phi$ is delivered by the formula
 \begin{align}\label{hadf1}
 d_{|\phi = \tilde \phi}(\Lambda_{F,\tau})[\psi]= \, &- \lambda_F^{\tau}[\tilde \phi]\binom{|F|-1}{\tau-1}
  \sum_{l \in F}  \bigg\{  \int_{\tilde \phi(\Omega)}(\lambda_F[\tilde \phi]v_l^2- |\nabla_Gv_l|^2 ) \mathrm{div} \left(\psi \circ  {\tilde\phi}^{(-1)}\right) \,dz  \\ \nonumber
  &\qquad\quad+\int_{\tilde \phi(\Omega)}(\nabla v_l)  (D(\psi \circ \tilde \phi^{(-1)})I_G^2 + I_G^2D(\psi \circ \tilde \phi^{(-1)})^t)(\nabla v_l)^t  \,dz\\\nonumber
  &\qquad\quad-\int_{\tilde \phi(\Omega)}2s|x|^{2s-2}|\nabla_y v_l|^2 x \cdot (\psi \circ \tilde \phi^{(-1)})_x \,dz
 \bigg\} \qquad \forall \psi \in L_{\Omega,O}.
  \end{align}
 \end{theorem}
 \begin{proof}
 We set $u_l := v_l \circ \tilde \phi$ for all $l \in F$ and we note that $\{u_l\}_{l \in F}$ is an orthonormal basis 
 in $\left(W^{1,2}_{G,0}(\Omega), Q_{G,\tilde \phi}\right)$ for the eigenspace corresponding to the eigenvalue 
 $\lambda_F^{-1}[\tilde \phi]$ of the operator $T_{G,\tilde \phi}$.  We recall that $M_{F,\tau}$ is 
 the operator defined in \eqref{Mdef}. By Lamberti and  Lanza de Cristoforis   \cite[Theorem 2.30]{LaLa04} it follows that
 \[
  d_{|\phi = \tilde \phi}(M_{F,\tau})[\psi]=  \lambda_F^{1-\tau}[\tilde \phi]\binom{|F|-1}{\tau-1}
  \sum_{l \in F} Q_{G,\tilde \phi}\left[ d_{|\phi = \tilde \phi} T_{G,\phi}[\psi][u_l],u_l\right]
  \qquad \forall \psi \in L_{\Omega,O}.
 \]
 Thus, exploiting formula \eqref{LM}, we have that
 \begin{align}\label{hadf2}
 &d_{|\phi = \tilde \phi}(\Lambda_{F,\tau})[\psi]\\\nonumber
  &= \left\{  d_{|\phi = \tilde \phi}M_{F,|F|-\tau}[\psi]M_{F,|F|}[\tilde \phi] - M_{F,|F|-\tau}[\tilde \phi]d_{|\phi = \tilde \phi}M_{F,|F|}[\psi]\right\} \lambda^{2|F|}_{F}[\tilde \phi] \\ \nonumber
  &=\left\{\lambda_F^{1-2|F|+\tau}[\tilde \phi]\binom{|F|-1}{|F|- \tau-1}
  -  \lambda_F^{1-2|F|+\tau}[\tilde \phi]\binom{|F|}{\tau}\right\}  
  \lambda^{2|F|}_{F}[\tilde \phi]  \sum_{l \in F}Q_{G,\tilde \phi}\left[d_{|\phi = \tilde \phi} T_{G,\phi}[\psi][u_l],u_l\right]  \\\nonumber
   &=-\lambda_F^{1+\tau}[\tilde \phi]\binom{|F|-1}{\tau-1}  
   \sum_{l \in F}Q_{G,\tilde \phi}\left[ d_{|\phi = \tilde \phi} T_{G,\phi}[\psi][u_l],u_l\right]
   \qquad\qquad \forall \psi \in L_{\Omega,O}.
 \end{align}
 Thus, we have to compute the term 
 $Q_{G,\tilde \phi}[  d_{|\phi = \tilde \phi} T_{G,\phi}[\psi][u_l],u_l].$
 By standard rules of calculus in Banach spaces, by the definition \eqref{Qdefphi} of $Q_{G,\phi}$, and 
 since every $u_l$ is an eigenfunction corresponding to the eigenvalue $\lambda_F^{-1}[\tilde \phi]$, 
 we have that 
 \begin{align*}
  Q_{G,\tilde \phi}&\left[d_{|\phi = \tilde \phi} T_{G,\phi}[\psi][u_l],u_l\right] \\
  &= Q_{G,\tilde \phi}\left[d_{|\phi = \tilde \phi} (-\Delta_{G,\phi}^{(-1)} \circ J_\phi \circ i)[\psi][u_l],u_l\right]\\
  &= Q_{G,\tilde \phi}\left[ -\Delta_{G,\tilde \phi}^{(-1)} \circ \left(d_{|\phi = \tilde \phi}(J_\phi \circ i)[\psi]\right)[u_l],u_l\right]+ Q_{G,\tilde \phi}\left[\left(d_{|\phi = \tilde \phi} (-\Delta_{G,\phi}^{(-1)})[\psi]\right) \circ J_{\tilde\phi }\circ i[u_l],u_l\right] \\
   &=   \left(d_{|\phi = \tilde \phi}(J_\phi \circ i)[\psi][u_l]\right)[u_l] 
   - \Delta_{G,\tilde \phi}\left[\Delta_{G, \tilde\phi}^{(-1)}\circ\left(d_{|\phi = \tilde \phi} (\Delta_{G,\phi})[\psi]\right)\circ \Delta_{G, \tilde\phi}^{(-1)} \circ J_{\tilde\phi} \circ i[u_l]\right][u_l] 
    \\
   &=   \left(d_{|\phi = \tilde \phi}(J_\phi \circ i)[\psi][u_l]\right)[u_l] 
 + \lambda_F^{-1}[\tilde \phi]\left(d_{|\phi = \tilde \phi} (\Delta_{G,\phi})[\psi]\right)[u_l][u_l] \qquad \forall \psi \in L_{\Omega,O}
 \end{align*}
(cf. Lamberti and Lanza de Cristoforis \cite[Lemma~3.26]{LaLa04}). Hence, in order to have an explicit representation of the differential, we need to compute the terms
 \[
 \left(d_{|\phi = \tilde \phi}(J_\phi \circ i)[\psi][u_l]\right)[u_l] \qquad \mbox{and} \qquad \left(d_{|\phi = \tilde \phi} (\Delta_{G,\phi})[\psi]\right)[u_l][u_l].
 \]
Standard rules of calculus in Banach spaces yield
 \begin{equation}\label{diffDetD}
 \left[\left(d_{|\phi = \tilde \phi} (\det D\phi )[\psi]\right)\circ {\tilde\phi}^{(-1)}\right]
 \det D {\tilde\phi}^{(-1)} = \mathrm{div} \left(\psi \circ  {\tilde\phi}^{(-1)}\right)  \qquad\qquad \forall \psi \in L_{\Omega,O}.
  \end{equation}
We note that the map from $\{f \in L^\infty(\Omega): \mathrm{essinf}_{\Omega} |f|>0\}$ to 
 $ L^\infty(\Omega)$ which takes $f$ to
 $|f|$ is differentiable and its differential at $f$ is the map from $ L^\infty(\Omega)$
to itself which maps $h$ to $\mathrm{sgn}(f)h$. By the above 
equality \eqref{diffDetD} and by a change of variable we obtain
\begin{align*}
 \left(d_{|\phi = \tilde \phi}(J_\phi \circ i)[\psi][u_l]\right)[u_l] &=
 \int_{\Omega}u_l^2d_{|\phi = \tilde \phi} (|\det D\phi|)[\psi]\,dz\\
 &= \int_{\tilde \phi(\Omega)}v_l^2d_{|\phi = \tilde \phi} ((|\det D\phi|)[\psi])\circ {\tilde\phi}^{(-1)} |\det D {\tilde\phi}^{(-1)}|\,dz\\
  &= \int_{\tilde \phi(\Omega)}v_l^2 \mathrm{div} \left(\psi \circ  {\tilde\phi}^{(-1)}\right) \,dz 
  \qquad\qquad \forall \psi \in L_{\Omega,O}.
\end{align*}
 Next, we turn to consider the shape differential of the term $((\Delta_{G,\phi})[\psi])[u_l][u_l]$.  By standard rules of calculus we have
\begin{align*}
d_{|\phi = \tilde \phi} (D\phi)^{-1}[\psi] = -(D\tilde \phi)^{-1}D\psi (D\tilde \phi)^{-1}
\qquad  \forall \psi \in L_{\Omega,O},
\end{align*} 
and 
\begin{align*}
d_{|\phi = \tilde \phi}I_G(\phi)[\psi] =  
\begin{pmatrix} 
 \mathbf{0}_{h \times h} & \mathbf{0}_{h \times k} \\
  \mathbf{0}_{k \times h}  & s|\tilde \phi_x|^{s-2} \tilde\phi_x \cdot \psi_x I_{k \times k}  
 \end{pmatrix}
    \qquad  \forall \psi \in L_{\Omega,O}.
\end{align*}
  Hence, 
 \begin{align*}
 \big(d_{|\phi = \tilde \phi}& (\Delta_{G,\phi})[\psi]\big)[u_l][u_l] \\
 =&\int_{\Omega}\nabla u_l   (D\tilde \phi)^{-1}D\psi (D\tilde \phi)^{-1} I_G(\tilde \phi) (\nabla u_l(D \tilde \phi)^{-1} I_G(\tilde \phi))^t|\det D\tilde \phi| \,dz \\
 &+\int_{\Omega}\nabla u_l (D\tilde \phi)^{-1} I_G(\tilde \phi) (\nabla u_l (D\tilde \phi)^{-1}D\psi (D \tilde \phi)^{-1} I_G(\tilde \phi))^t|\det D\tilde \phi| \,dz \\
  &-\int_{\Omega}\nabla u_l  (D\tilde \phi)^{-1} I_G(\tilde \phi) (\nabla u_l(D \tilde \phi)^{-1} I_G(\tilde \phi))^t(d_{|\phi = \tilde \phi}|\det D\phi|)[\psi] \,dz \\
  &-\int_{\Omega}\nabla u_l  (D\tilde \phi)^{-1} (d_{|\phi = \tilde \phi}I_G(\phi)[\psi]) (\nabla u_l(D \tilde \phi)^{-1} I_G(\tilde \phi))^t|\det D\tilde \phi| \,dz \\
    &-\int_{\Omega}\nabla u_l  (D\tilde \phi)^{-1} I_G(\phi) (\nabla u_l(D \tilde \phi)^{-1}  (d_{|\phi = \tilde \phi}I_G(\phi)[\psi])(\tilde \phi))^t|\det D\tilde \phi| \,dz \\
  =& \int_{\tilde \phi(\Omega)}(\nabla v_l)  (D(\psi \circ \tilde \phi^{(-1)})I_G)(\nabla_Gv_l)^t  \,dz \\
  &+\int_{\tilde \phi(\Omega)}(\nabla_G v_l)  (D(\psi \circ \tilde \phi^{(-1)})I_G)^t(\nabla v_l)^t  \,dz\\
  &-\int_{\tilde \phi(\Omega)}|\nabla_Gv_l|^2     
     \mathrm{div} \left(\psi \circ  {\tilde\phi}^{(-1)}\right)\,dz\\
  &-2s\int_{\tilde \phi(\Omega)}|x|^{2s-2}|\nabla_y v_l|^2 x \cdot (\psi \circ \tilde \phi^{(-1)})_x \,dz\\
  =&\int_{\tilde \phi(\Omega)}(\nabla v_l)  \Big(D(\psi \circ \tilde \phi^{(-1)})I_G^2 + I_G^2D(\psi \circ \tilde \phi^{(-1)})^t\Big)(\nabla v_l)^t  \,dz \\
  &-\int_{\tilde \phi(\Omega)}|\nabla_Gv_l|^2     
     \mathrm{div} \left(\psi \circ  {\tilde\phi}^{(-1)}\right)\,dz\\
  &-\int_{\tilde \phi(\Omega)}2s|x|^{2s-2}|\nabla_y v_l|^2 x \cdot (\psi \circ \tilde \phi^{(-1)})_x \,dz
  \qquad \forall \psi \in  L_{\Omega,O}.
 \end{align*}
 Accordingly, we   have proved that
 \begin{align*}
 Q_{G,\tilde \phi}\Big[  d_{|\phi = \tilde \phi} T_{G,\phi}&[\psi][u_l],u_l\Big]
 \\= &
 \int_{\tilde \phi(\Omega)}v_l^2 \mathrm{div} \left(\psi \circ  {\tilde\phi}^{(-1)}\right) \,dz \\ \nonumber
 &+ \lambda_F^{-1}[\tilde \phi] \int_{\tilde \phi(\Omega)}(\nabla v_l)  \Big(D(\psi \circ \tilde \phi^{(-1)})I_G^2 + I_G^2D(\psi \circ \tilde \phi^{(-1)})^t\Big)(\nabla v_l)^t  \,dz \\ \nonumber
 & - \lambda_F^{-1}[\tilde \phi] \int_{\tilde \phi(\Omega)}|\nabla_Gv_l|^2     
     \mathrm{div} \left(\psi \circ  {\tilde\phi}^{(-1)}\right)\,dz \\ \nonumber
  & - \lambda_F^{-1}[\tilde \phi] \int_{\tilde \phi(\Omega)}2s|x|^{2s-2}|\nabla_y v_l|^2 x \cdot (\psi \circ \tilde \phi^{(-1)})_x \,dz
      \qquad \forall \psi \in L_{\Omega,O}.
 \end{align*}
 Putting together all the above equalities one verifies that formula \eqref{hadf1} holds.
 \end{proof}
 
Now, our aim is to rewrite formula \eqref{hadf1} in a simpler form and obtain a Grushin analog of the classical Hadamard formula. To achieve this goal, we 
 prove an intermediate technical lemma where we provide a suitable representation formula for  $Q_{G,\tilde \phi}\Big[  d_{|\phi = \tilde \phi} T_{G,\phi}[\psi][u_1],u_2\Big]$, where $u_1$, $u_2$ are two eigenfunctions associated 
  with the same eigenvalue.   The following lemma is the analog in the Grushin setting of Lanza de Cristoforis and Lamberti 
 \cite[Lemma 3.26]{LaLa04} for the standard Laplacian. We note that, although the idea behind the proof is the same, 
 the Grushin case requires a careful and not straightforward analysis of several terms which do not appear in the standard case. For this reason we include a detailed proof. 
 \begin{lemma}\label{Qformula}
Let  $\Omega$ and  $O$ be as in \eqref{Omega_def}.
Let $\tilde \phi \in \mathcal{A}_{\Omega,O}$. Suppose that $\tilde \phi(\Omega)$ is of class $C^1$.
 Let $v_1, v_2 \in W^{1,2}_{G,0}(\tilde \phi(\Omega))$  be two eigenfunctions corresponding to an 
 eigenvalue $\lambda[\tilde \phi]$ of \eqref{wsp}.
Suppose that $v_1,v_2 \in W^{1,2}_0(\tilde \phi(\Omega)) \cap W^{2,2}(\tilde\phi(\Omega))$. Let $u_1:=v_1 \circ \tilde \phi$, $u_2:=v_2 \circ \tilde \phi$.
 Then
 \begin{align}\label{Qformula1}
 Q_{G,\tilde \phi}\Big[  d_{|\phi = \tilde \phi} T_{G,\phi}[\psi][u_1],u_2\Big]  = \lambda[\tilde \phi]^{-1}\int_{\partial\tilde \phi(\Omega)}
  (\psi \circ  {\tilde\phi}^{(-1)} \mathbf{n}^t)  \frac{\partial v_1}{\partial \mathbf{n}}\frac{\partial v_2}{\partial \mathbf{n}}|\mathbf{n}_G|^2\,d\sigma \qquad \forall \psi \in L_{\Omega,O},
 \end{align}
  where $\mathbf{n}$ denotes the outward unit normal field to $\partial \tilde \phi (\Omega)$ and 
  \begin{equation}\label{grushinnormal}
 \mathbf{n}_G :=  \mathbf{n}\, I_G=(\mathbf{n}_x,|x|^s\mathbf{n}_y) \qquad \mbox{ on } \partial \tilde \phi (\Omega).
 \end{equation}
 \end{lemma}
 \begin{proof}
 We fix $\psi \in L_{\Omega,O}$ and for the sake of simplicity we set  $\omega := \psi \circ  {\tilde\phi}^{(-1)}$.
 A minor modification of the proof of Theorem \ref{hadf} shows that 
  \begin{align}\label{Qformula1'}
 Q_{G,\tilde \phi}\Big[d_{|\phi = \tilde \phi} T_{G,\phi}&[\psi][u_1],u_2\Big] 
 \\ \nonumber
 = &\int_{\tilde \phi(\Omega)}v_1v_2 \mathrm{div} \left(\omega\right) \,dz \\ \nonumber
 &+ \lambda^{-1}[\tilde \phi] \int_{\tilde \phi(\Omega)}(\nabla v_1)  \Big((D\omega) I_G^2 + I_G^2(D\omega)^t\Big)(\nabla v_2)^t  \,dz \\ \nonumber
 & - \lambda^{-1}[\tilde \phi] \int_{\tilde \phi(\Omega)}(\nabla_Gv_1)(\nabla_Gv_2)^t    
     \mathrm{div} \left(\omega\right)\,dz \\ \nonumber
  & - \lambda^{-1}[\tilde \phi] \int_{\tilde \phi(\Omega)}2s|x|^{2s-2}(\nabla_y v_1)(\nabla_y v_2)^t x \cdot \omega_x \,dz
      \qquad \forall \psi \in L_{\Omega,O}.
 \end{align}
  We start by considering the second term in the right hand side of formula  \eqref{Qformula1'}. 
 A simple computation shows that
 \begin{align} \label{hadform1}
\int_{\tilde \phi(\Omega)}\nabla v_1
  \Big((D\omega)& I_G^2 + I_G^2 (D\omega)^t \Big)
  \left( \nabla v_2 \right)^t \,dz \\ \nonumber
  =&\int_{\tilde \phi(\Omega)}\nabla_x v_1
  \Big(D_x\omega_x  + (D_x\omega_x)^t \Big)
  \left( \nabla_x v_2 \right)^t \,dz \\ \nonumber
    &+\int_{\tilde \phi(\Omega)}|x|^{2s}\nabla_y v_1
  \Big(D_y\omega_y + (D_y\omega_y)^t \Big)
  \left( \nabla_y v_2 \right)^t \,dz \\\nonumber
  &+\int_{\tilde \phi(\Omega)} \nabla_y v_1
  \Big(D_x\omega_y + |x|^{2s}(D_y\omega_x)^t \Big)
  \left( \nabla_x v_2 \right)^t \,dz\\\nonumber
  &+\int_{\tilde \phi(\Omega)} \nabla_y v_2
  \Big(D_x\omega_y + |x|^{2s}(D_y\omega_x)^t \Big)
  \left( \nabla_x v_1 \right)^t \,dz.
 \end{align}
 We consider the second term in the right hand side of equation \eqref{hadform1}. By the Divergence Theorem we have that
\begin{align*}
\int_{\tilde \phi(\Omega)}|x|^{2s}\nabla_y v_1&
  \Big(D_y\omega_y + (D_y\omega_y)^t \Big)
  \left( \nabla_y v_2 \right)^t \,dz\\
  =&\sum_{i,j=1}^{k}\int_{\tilde \phi(\Omega)}\left(|x|^{2s} 
  \frac{\partial \omega_{y_i}}{\partial y_j}  \frac{\partial v_1}{\partial y_i}  \frac{\partial v_2}{\partial y_j}
  +|x|^{2s} 
  \frac{\partial \omega_{y_j}}{\partial y_i}  \frac{\partial v_1}{\partial y_i}  \frac{\partial v_2}{\partial y_j}\right)\,dz  \\
  =&-\sum_{i,j=1}^{k}\int_{\tilde \phi(\Omega)}\left(|x|^{2s} 
 \omega_{y_i} \frac{\partial^2 v_1}{\partial y_j\partial y_i}  \frac{\partial v_2}{\partial y_j}
 +|x|^{2s} 
 \omega_{y_i} \frac{\partial v_1}{\partial y_i}  \frac{\partial^2 v_2}{\partial y_j^2}
\right)\,dz\\
&-\sum_{i,j=1}^{k}\int_{\tilde \phi(\Omega)}\left(|x|^{2s} 
 \omega_{y_j} \frac{\partial^2 v_1}{\partial y_i^2}  \frac{\partial v_2}{\partial y_j}
 +|x|^{2s} 
 \omega_{y_j} \frac{\partial v_1}{\partial y_i}  \frac{\partial^2 v_2}{\partial y_i\partial y_j}
\right)\,dz\\
&+\sum_{i,j=1}^{k}\int_{\partial\tilde \phi(\Omega)}\left(|x|^{2s} 
  \omega_{y_i} \mathbf{n}_{y_j} \frac{\partial v_1}{\partial y_i}  \frac{\partial v_2}{\partial y_j}
  +|x|^{2s} 
   \omega_{y_j} \mathbf{n}_{y_i} \frac{\partial v_1}{\partial y_i}  \frac{\partial v_2}{\partial y_j}\right)\,d\sigma_z.
\end{align*}
Clearly 
\begin{align*}
\sum_{i,j=1}^{k}\int_{\tilde \phi(\Omega)}|x|^{2s} 
 \omega_{y_i} \frac{\partial v_1}{\partial y_i}  \frac{\partial^2 v_2}{\partial y_j^2}
\,dz
= \int_{\tilde \phi(\Omega)}|x|^{2s}\Delta_yv_2 (\nabla_y v_1)\omega_y^t\,dz,
\end{align*}
and 
\begin{align*}
\sum_{i,j=1}^{k}\int_{\tilde \phi(\Omega)}|x|^{2s} 
 \omega_{y_j} \frac{\partial^2 v_1}{\partial y_i^2}  \frac{\partial v_2}{\partial y_j}
\,dz= \int_{\tilde \phi(\Omega)}|x|^{2s}\Delta_yv_1 (\nabla_y v_2)\omega_y^t\,dz.
\end{align*}
Moreover, 
 \begin{align*}
\sum_{i,j=1}^{k}\int_{\tilde \phi(\Omega)}|x|^{2s}& \omega_{y_i} \frac{\partial^2 v_1}{\partial y_j\partial y_i}  \frac{\partial v_2}{\partial y_j}\,dz +
\sum_{i,j=1}^{k}\int_{\tilde \phi(\Omega)}|x|^{2s} 
 \omega_{y_j} \frac{\partial v_1}{\partial y_i}  \frac{\partial^2 v_2}{\partial y_i\partial y_j}\,dz\\
 =&\sum_{i,j=1}^{k}\int_{\tilde \phi(\Omega)}|x|^{2s} \omega_{y_i} \frac{\partial^2 v_1}{\partial y_j\partial y_i}  \frac{\partial v_2}{\partial y_j}\,dz +
\sum_{i,j=1}^{k}\int_{\tilde \phi(\Omega)}|x|^{2s} 
 \omega_{y_i} \frac{\partial v_1}{\partial y_j}  \frac{\partial^2 v_2}{\partial y_j\partial y_i}\,dz\\
 =&\sum_{i,j=1}^{k}\int_{\tilde \phi(\Omega)}|x|^{2s} \omega_{y_i}
 \frac{\partial}{\partial y_i}\left( \frac{\partial v_1}{\partial y_j}  \frac{\partial v_2}{\partial y_j}\right)\,dz\\
 =&-\sum_{i,j=1}^{k}\int_{\tilde \phi(\Omega)}|x|^{2s} 
 \frac{\partial\omega_{y_i}}{\partial y_i} \frac{\partial v_1}{\partial y_j}  \frac{\partial v_2}{\partial y_j}\,dz
 +\sum_{i,j=1}^{k}\int_{\partial\tilde \phi(\Omega)}|x|^{2s} 
  \omega_{y_i}\mathbf{n}_{y_i}
 \frac{\partial v_1}{\partial y_j}  \frac{\partial v_2}{\partial y_j}\,d\sigma_z\\
  =&-\int_{\tilde \phi(\Omega)}|x|^{2s} (\nabla_yv_1)(\nabla_yv_2)^t\mathrm{div}_y\omega_y\,dz
  +\int_{\partial\tilde \phi(\Omega)}|x|^{2s} (\omega_y \mathbf{n}_y^t)((\nabla_yv_1)(\nabla_yv_2)^t)\,d\sigma_z,
\end{align*}
and 
\begin{align*}
\sum_{i,j=1}^{k}\int_{\partial\tilde \phi(\Omega)}&\left(|x|^{2s} 
  \omega_{y_i} \mathbf{n}_{y_j} \frac{\partial v_1}{\partial y_i}  \frac{\partial v_2}{\partial y_j}
  +|x|^{2s} 
   \omega_{y_j} \mathbf{n}_{y_i} \frac{\partial v_1}{\partial y_i}  \frac{\partial v_2}{\partial y_j}\right)\,d\sigma_z\\
   =&\int_{\partial\tilde \phi(\Omega)}\left(|x|^{2s} 
  (\omega_y\nabla_yv_1^t)(\mathbf{n}_y\nabla_yv_2^t)+|x|^{2s} 
  (\omega_y\nabla_yv_2^t)(\mathbf{n}_y\nabla_yv_1^t)\right) \,d\sigma_z.
\end{align*}
Accordingly, the second term in the right hand side of equation \eqref{hadform1} equals
\begin{align*}
\int_{\tilde \phi(\Omega)}|x|^{2s}\nabla_y v_1
  \Big(&D_y\omega_y + (D_y\omega_y)^t \Big)
  \left( \nabla_y v_2 \right)^t \,dz\\
  =&- \int_{\tilde \phi(\Omega)}|x|^{2s}\Delta_yv_2 (\nabla_y v_1)\omega_y^t\,dz\\
  &- \int_{\tilde \phi(\Omega)}|x|^{2s}\Delta_yv_1 (\nabla_y v_2)\omega_y^t\,dz\\
  &+\int_{\tilde \phi(\Omega)}|x|^{2s} (\nabla_yv_1)(\nabla_yv_2)^t\mathrm{div}_y\omega_y\,dz\\
  &-\int_{\partial\tilde \phi(\Omega)}|x|^{2s} (\omega_y \mathbf{n}_y^t)((\nabla_yv_1)(\nabla_yv_2)^t)\,d\sigma_z\\
  &+\int_{\partial\tilde \phi(\Omega)}\left(|x|^{2s} 
  (\omega_y\nabla_yv_1^t)(\mathbf{n}_y\nabla_yv_2^t)+|x|^{2s} 
  (\omega_y\nabla_yv_2^t)(\mathbf{n}_y\nabla_yv_1^t)\right) \,d\sigma_z.
\end{align*}
Similarly,  the first term in the right hand side of equation \eqref{hadform1} equals
\begin{align*}
\int_{\tilde \phi(\Omega)}\nabla_x v_1 
  \Big(D_x\omega_x& + (D_x\omega_x)^t \Big)
  \left( \nabla_x v_2 \right)^t \,dz\\
  =&- \int_{\tilde \phi(\Omega)}\Delta_xv_2(\nabla_x v_1)\omega_x^t\,dz\\
  &- \int_{\tilde \phi(\Omega)}\Delta_xv_1 (\nabla_x v_2)\omega_x^t\,dz\\
  &+\int_{\tilde \phi(\Omega)} (\nabla_xv_1)(\nabla_xv_2)^t\mathrm{div}_x\omega_x\,dz\\
  &-\int_{\partial\tilde \phi(\Omega)} (\omega_x \mathbf{n}_x^t)((\nabla_xv_1)(\nabla_xv_2)^t)\,d\sigma_z\\
  &+\int_{\partial\tilde \phi(\Omega)}\left(
  (\omega_x\nabla_xv_1^t)(\mathbf{n}_x\nabla_xv_2^t)+
  (\omega_x\nabla_xv_2^t)(\mathbf{n}_x\nabla_xv_1^t)\right) \,d\sigma_z.
\end{align*}
Since $v_1,v_2$ are eigenfunctions we  have
\begin{align*}
-\int_{\tilde \phi(\Omega)}\Delta_xv_1 (\nabla_x v_2)\omega_x^t\,dz &  -\int_{\tilde \phi(\Omega)}|x|^{2s}\Delta_yv_1(\nabla_y v_2)\omega_y^t\,dz\\
=& -\; \int_{\tilde \phi(\Omega)}\Delta_{G}v_1 (\nabla v_2)\omega^t\,dz +
   \int_{\tilde \phi(\Omega)}\Delta_xv_1 (\nabla_y v_2)\omega_y^t\,dz\\
 &+\int_{\tilde \phi(\Omega)}|x|^{2s}\Delta_yv_1 (\nabla_x v_2)\omega_x^t\,dz\\
 =&\,\lambda[\tilde\phi] \int_{\tilde \phi(\Omega)}v_1 (\nabla v_2)\omega^t\,dz
 +  \int_{\tilde \phi(\Omega)}\Delta_xv_1 (\nabla_y v_2)\omega_y^t\,dz\\
 &+\int_{\tilde \phi(\Omega)}|x|^{2s}\Delta_yv_1 (\nabla_x v_2)\omega_x^t\,dz,
\end{align*}
and
\begin{align*}
-\int_{\tilde \phi(\Omega)}\Delta_xv_2 (\nabla_x v_1)\omega_x^t\,dz &  -\int_{\tilde \phi(\Omega)}|x|^{2s}\Delta_yv_2(\nabla_y v_1)\omega_y^t\,dz\\
 =&\,\lambda[\tilde\phi] \int_{\tilde \phi(\Omega)}v_2 (\nabla v_1)\omega^t\,dz
 +  \int_{\tilde \phi(\Omega)}\Delta_xv_2 (\nabla_y v_1)\omega_y^t\,dz\\
 &+\int_{\tilde \phi(\Omega)}|x|^{2s}\Delta_yv_2 (\nabla_x v_1)\omega_x^t\,dz,
\end{align*}
and accordingly, 
\begin{align*}
-\int_{\tilde \phi(\Omega)}\Delta_xv_1 (\nabla_x v_2)\omega_x^t\,dz &  -\int_{\tilde \phi(\Omega)}|x|^{2s}\Delta_yv_1(\nabla_y v_2)\omega_y^t\,dz\\
-\int_{\tilde \phi(\Omega)}\Delta_xv_2 (\nabla_x v_1)\omega_x^t\,dz &  -\int_{\tilde \phi(\Omega)}|x|^{2s}\Delta_yv_2(\nabla_y v_1)\omega_y^t\,dz\\
 =&\,\lambda[\tilde\phi] \int_{\tilde \phi(\Omega)}\nabla (v_1 v_2)\omega^t\,dz
  +  \int_{\tilde \phi(\Omega)}\Delta_xv_1 (\nabla_y v_2)\omega_y^t\,dz\\
 &+\int_{\tilde \phi(\Omega)}|x|^{2s}\Delta_yv_1 (\nabla_x v_2)\omega_x^t\,dz
 +  \int_{\tilde \phi(\Omega)}\Delta_xv_2 (\nabla_y v_1)\omega_y^t\,dz\\
 &+\int_{\tilde \phi(\Omega)}|x|^{2s}\Delta_yv_2 (\nabla_x v_1)\omega_x^t\,dz\\
 =&\,-\lambda[\tilde\phi] \int_{\tilde \phi(\Omega)} v_1 v_2\mathrm{div}(\omega)\,dz
  +  \int_{\tilde \phi(\Omega)}\Delta_xv_1 (\nabla_y v_2)\omega_y^t\,dz\\
 &+\int_{\tilde \phi(\Omega)}|x|^{2s}\Delta_yv_1 (\nabla_x v_2)\omega_x^t\,dz
 +  \int_{\tilde \phi(\Omega)}\Delta_xv_2 (\nabla_y v_1)\omega_y^t\,dz\\
 &+\int_{\tilde \phi(\Omega)}|x|^{2s}\Delta_yv_2 (\nabla_x v_1)\omega_x^t\,dz.
\end{align*}
Thus,
 \begin{align}\label{hadform2}
 Q_{G,\tilde \phi}\left[  d_{|\phi = \tilde \phi} T_{G,\phi}[\psi][u_1],u_2\right] = &\,
\lambda^{-1}[\tilde \phi] \int_{\tilde \phi(\Omega)}\Delta_xv_1 (\nabla_y v_2)\omega_y^t\,dz\\ \nonumber
&+
\lambda^{-1}[\tilde \phi] \int_{\tilde \phi(\Omega)}\Delta_xv_2 (\nabla_y v_1)\omega_y^t\,dz\\ \nonumber
  &+ \lambda^{-1}[\tilde \phi]\int_{\tilde \phi(\Omega)} (\nabla_xv_1)(\nabla_xv_2)^t\mathrm{div}_x\omega_x\,dz\\\nonumber
  &- \lambda^{-1}[\tilde \phi]\int_{\partial\tilde \phi(\Omega)} (\omega_x \mathbf{n}_x^t)(\nabla_xv_1)(\nabla_xv_2)^t\,d\sigma_z\\\nonumber
  &+ \lambda^{-1}[\tilde \phi]\int_{\partial\tilde \phi(\Omega)}\left(
  (\omega_x\nabla_xv_1^t)(\mathbf{n}_x\nabla_xv_2^t)+
  (\omega_x\nabla_xv_2^t)(\mathbf{n}_x\nabla_xv_1^t)\right) \,d\sigma_z\\\nonumber
    &+ \lambda^{-1}[\tilde \phi] \int_{\tilde \phi(\Omega)}|x|^{2s}\Delta_yv_1 (\nabla_x v_2)\omega_x^t\,dz\\\nonumber
      &+ \lambda^{-1}[\tilde \phi] \int_{\tilde \phi(\Omega)}|x|^{2s}\Delta_yv_2 (\nabla_x v_1)\omega_x^t\,dz\\\nonumber
  &+ \lambda^{-1}[\tilde \phi]\int_{\tilde \phi(\Omega)}|x|^{2s} (\nabla_yv_1)(\nabla_yv_2)^t\mathrm{div}_y\omega_y\,dz\\\nonumber
  &- \lambda^{-1}[\tilde \phi]\int_{\partial\tilde \phi(\Omega)}|x|^{2s} (\omega_y \mathbf{n}_y^t)(\nabla_yv_1)(\nabla_y v_2)^t\,d\sigma_z\\\nonumber
  &+ \lambda^{-1}[\tilde \phi]\int_{\partial\tilde \phi(\Omega)}\left(|x|^{2s} 
  (\omega_y\nabla_yv_1^t)(\mathbf{n}_y\nabla_yv_2^t)+|x|^{2s} 
  (\omega_y\nabla_yv_2^t)(\mathbf{n}_y\nabla_yv_1^t)\right) \,d\sigma_z\\\nonumber
   &+\lambda^{-1}[\tilde\phi]\int_{\tilde \phi(\Omega)} \nabla_y v_1
  \Big(D_x\omega_y + |x|^{2s}(D_y\omega_x)^t \Big)
  \left( \nabla_x v_2 \right)^t \,dz\\\nonumber
  &+\lambda^{-1}[\tilde \phi]\int_{\tilde \phi(\Omega)} \nabla_y v_2
  \Big(D_x\omega_y + |x|^{2s}(D_y\omega_x)^t \Big)
  \left( \nabla_x v_1 \right)^t \,dz\\\nonumber
 & - \lambda^{-1}[\tilde \phi] \int_{\tilde \phi(\Omega)}(\nabla_Gv_1)(\nabla_Gv_2)^t   
     \mathrm{div}\, \omega\,dz \\\nonumber
  & - \lambda^{-1}[\tilde \phi] \int_{\tilde \phi(\Omega)}2s|x|^{2s-2}(\nabla_y v_1)(\nabla_y v_2)^t x \cdot\omega_x \,dz\\\nonumber
  =&\,
\lambda^{-1}[\tilde \phi] \int_{\tilde \phi(\Omega)}\Delta_xv_1 (\nabla_y v_2)\omega_y^t\,dz\\ \nonumber
&+
\lambda^{-1}[\tilde \phi] \int_{\tilde \phi(\Omega)}\Delta_xv_2 (\nabla_y v_1)\omega_y^t\,dz\\ \nonumber
  &- \lambda^{-1}[\tilde \phi]\int_{\tilde \phi(\Omega)} (\nabla_xv_1)(\nabla_xv_2)^t\mathrm{div}_y\omega_y\,dz\\\nonumber
  &- \lambda^{-1}[\tilde \phi]\int_{\partial\tilde \phi(\Omega)} (\omega_x \mathbf{n}_x^t)(\nabla_xv_1)(\nabla_xv_2)^t\,d\sigma_z\\\nonumber
  &+ \lambda^{-1}[\tilde \phi]\int_{\partial\tilde \phi(\Omega)}\left(
  (\omega_x\nabla_xv_1^t)(\mathbf{n}_x\nabla_xv_2^t)+
  (\omega_x\nabla_xv_2^t)(\mathbf{n}_x\nabla_xv_1^t)\right) \,d\sigma_z\\\nonumber
    &+ \lambda^{-1}[\tilde \phi] \int_{\tilde \phi(\Omega)}|x|^{2s}\Delta_yv_1 (\nabla_x v_2)\omega_x^t\,dz\\\nonumber
      &+ \lambda^{-1}[\tilde \phi] \int_{\tilde \phi(\Omega)}|x|^{2s}\Delta_yv_2 (\nabla_x v_1)\omega_x^t\,dz\\\nonumber
  &- \lambda^{-1}[\tilde \phi]\int_{\tilde \phi(\Omega)}|x|^{2s} (\nabla_yv_1)(\nabla_y v_2)^t\mathrm{div}_x\omega_x\,dz\\\nonumber
  &- \lambda^{-1}[\tilde \phi]\int_{\partial\tilde \phi(\Omega)}|x|^{2s} (\omega_y \mathbf{n}_y^t)(\nabla_yv_1)(\nabla_yv_2)^t\,d\sigma_z\\\nonumber
 &+ \lambda^{-1}[\tilde \phi]\int_{\partial\tilde \phi(\Omega)}\left(|x|^{2s} 
  (\omega_y\nabla_yv_1^t)(\mathbf{n}_y\nabla_yv_2^t)+|x|^{2s} 
  (\omega_y\nabla_yv_2^t)(\mathbf{n}_y\nabla_yv_1^t)\right) \,d\sigma_z\\\nonumber
   &+ \lambda^{-1}[\tilde \phi]\int_{\tilde \phi(\Omega)} \nabla_y v_1
  \Big(D_x\omega_y + |x|^{2s}(D_y\omega_x)^t \Big)
  \left( \nabla_x v_2 \right)^t \,dz\\\nonumber
  &+ \lambda^{-1}[\tilde \phi]\int_{\tilde \phi(\Omega)} \nabla_y v_2
  \Big(D_x\omega_y + |x|^{2s}(D_y\omega_x)^t \Big)
  \left( \nabla_x v_1 \right)^t \,dz\\
        & - \lambda^{-1}[\tilde \phi] \int_{\tilde \phi(\Omega)}2s|x|^{2s-2}(\nabla_y v_1)(\nabla_y v_2)^t x \cdot \omega_x \,dz.\nonumber
 \end{align}
We now set
 \begin{align*}
I_1&:= \lambda^{-1}[\tilde \phi] \int_{\tilde \phi(\Omega)}\Delta_xv_1 (\nabla_y v_2)\omega_y^t\,dz\\ \nonumber
&\quad+
\lambda^{-1}[\tilde \phi] \int_{\tilde \phi(\Omega)}\Delta_xv_2 (\nabla_y v_1)\omega_y^t\,dz\\ \nonumber
  &\quad- \lambda^{-1}[\tilde \phi]\int_{\tilde \phi(\Omega)} (\nabla_xv_1)(\nabla_xv_2)^t\mathrm{div}_y\omega_y\,dz\\
  &\quad+ \lambda^{-1}[\tilde \phi] \int_{\tilde \phi(\Omega)}|x|^{2s}\Delta_yv_1 (\nabla_x v_2)\omega_x^t\,dz\\\nonumber
      &\quad+ \lambda^{-1}[\tilde \phi] \int_{\tilde \phi(\Omega)}|x|^{2s}\Delta_yv_2 (\nabla_x v_1)\omega_x^t\,dz\\\nonumber
  &\quad- \lambda^{-1}[\tilde \phi]\int_{\tilde \phi(\Omega)}|x|^{2s} (\nabla_yv_1)(\nabla_yv_2)^t\mathrm{div}_x\omega_x\,dz\\
    &\quad+ \lambda^{-1}[\tilde \phi]\int_{\tilde \phi(\Omega)} \nabla_y v_1
  \Big(D_x\omega_y + |x|^{2s}(D_y\omega_x)^t \Big)
  \left( \nabla_x v_2 \right)^t \,dz\\\nonumber
  &\quad+ \lambda^{-1}[\tilde \phi]\int_{\tilde \phi(\Omega)} \nabla_y v_2
  \Big(D_x\omega_y + |x|^{2s}(D_y\omega_x)^t \Big)
  \left( \nabla_x v_1 \right)^t \,dz\\
      &\quad - \lambda^{-1}[\tilde \phi] \int_{\tilde \phi(\Omega)}2s|x|^{2s-2}(\nabla_y v_1)(\nabla_y v_2)^t x \omega_x^t \,dz\, ,\\
I_2&:= - \lambda^{-1}[\tilde \phi]\int_{\partial\tilde \phi(\Omega)} (\omega_x \mathbf{n}_x^t)(\nabla_x v_1)(\nabla_x v_2)^t\,d\sigma_z \\ & \quad- \lambda^{-1}[\tilde \phi]\int_{\partial\tilde \phi(\Omega)}|x|^{2s} (\omega_y \mathbf{n}_y^t)(\nabla_y v_1)(\nabla_y v_2)^t\,d\sigma_z\, ,\\
I_3&:= \lambda^{-1}[\tilde \phi]\int_{\partial\tilde \phi(\Omega)}\left(
  (\omega_x\nabla_xv_1^t)(\mathbf{n}_x\nabla_xv_2^t)+
  (\omega_x\nabla_xv_2^t)(\mathbf{n}_x\nabla_xv_1^t)\right) \,d\sigma_z\\\nonumber 
  &\quad+ \lambda^{-1}[\tilde \phi]\int_{\partial\tilde \phi(\Omega)}\left(|x|^{2s} 
  (\omega_y\nabla_yv_1^t)(\mathbf{n}_y\nabla_yv_2^t)+|x|^{2s} 
  (\omega_y\nabla_yv_2^t)(\mathbf{n}_y\nabla_yv_1^t)\right) \,d\sigma_z\, .
 \end{align*}
As a consequence, equality  \eqref{hadform2} reads as
 \begin{equation}\label{hadform2bis}
Q_{G,\tilde \phi}\left[  d_{|\phi = \tilde \phi} T_{G,\phi}[\psi][u_1],u_2\right] = I_1+I_2+I_3\, .
 \end{equation}
 Now, we rewrite the terms $I_2$ and $I_3$. We first consider $I_2$. 
\begin{align*}
I_2=&- \lambda^{-1}[\tilde \phi]\int_{\partial\tilde \phi(\Omega)} (\omega_x \mathbf{n}_x^t)(\nabla_x v_1)(\nabla_x v_2)^t\,d\sigma_z -  \lambda^{-1}[\tilde \phi]\int_{\partial\tilde \phi(\Omega)}|x|^{2s} (\omega_y \mathbf{n}_y^t)(\nabla_y v_1)(\nabla_y v_2)^t\,d\sigma_z\\
=&- \lambda^{-1}[\tilde \phi]\int_{\partial\tilde \phi(\Omega)} (\omega_x \mathbf{n}_x^t)(\nabla_G v_1)(\nabla_G v_2)^t\,d\sigma_z
+\lambda^{-1}[\tilde \phi]\int_{\partial\tilde \phi(\Omega)}|x|^{2s} (\omega_x \mathbf{n}_x^t)(\nabla_y v_1)(\nabla_y v_2)^t\,d\sigma_z\\
&- \lambda^{-1}[\tilde \phi]\int_{\partial\tilde \phi(\Omega)} (\omega_y \mathbf{n}_y^t)(\nabla_G v_1)(\nabla_G v_2)^t\,d\sigma_z
+\lambda^{-1}[\tilde \phi]\int_{\partial\tilde \phi(\Omega)} (\omega_y \mathbf{n}_y^t)(\nabla_x v_1)(\nabla_x v_2)^t\,d\sigma_z\\
=&- \lambda^{-1}[\tilde \phi]\int_{\partial\tilde \phi(\Omega)} (\omega \mathbf{n}^t)(\nabla_G v_1)(\nabla_G v_2)^t\,d\sigma_z
+\lambda^{-1}[\tilde \phi]\int_{\partial\tilde \phi(\Omega)}|x|^{2s} (\omega_x \mathbf{n}_x^t)(\nabla_y v_1)(\nabla_y v_2)^t\,d\sigma_z\\
&+\lambda^{-1}[\tilde \phi]\int_{\partial\tilde \phi(\Omega)} (\omega_y \mathbf{n}_y^t)(\nabla_x v_1)(\nabla_x v_2)^t\,d\sigma_z.
\end{align*}
We now consider the last two summands of the right hand side of the previous equality. We have: 
\begin{align*}
\lambda^{-1}[\tilde \phi]&\int_{\partial \tilde \phi(\Omega)}|x|^{2s} (\omega_x \mathbf{n}_x^t)(\nabla_y v_1)(\nabla_y v_2)^t\,d\sigma_z \\
=&\; \lambda^{-1}[\tilde \phi]\int_{\tilde \phi(\Omega)}\mathrm{div}_x \big(|x|^{2s}(\nabla_y v_1)(\nabla_y v_2)^t\omega_x\big)\,dz\\
=& \;\lambda^{-1}[\tilde \phi]\int_{\tilde \phi(\Omega)}|x|^{2s}(\nabla_y v_1)(\nabla_y v_2)^t\mathrm{div}_x \omega_x\,dz 
+ \lambda^{-1}[\tilde \phi]\int_{\tilde \phi(\Omega)}\nabla_x \Big(|x|^{2s}(\nabla_y v_1)(\nabla_y v_2)^t\Big)\omega_x^t\,dz \\
= & \;\lambda^{-1}[\tilde \phi]\int_{\tilde \phi(\Omega)}|x|^{2s}(\nabla_y v_1)(\nabla_y v_2)^t\mathrm{div}_x \omega_x\,dz
+ \lambda^{-1}[\tilde \phi]\int_{\tilde \phi(\Omega)}|x|^{2s}\nabla_y v_1\Big(D_x \nabla_y v_2\Big)\omega^t_x\,dz\\
&+ \lambda^{-1}[\tilde \phi]\int_{\tilde \phi(\Omega)}|x|^{2s}\nabla_y v_2\Big(D_x \nabla_y v_1\Big)\omega^t_x\,dz+ \lambda^{-1}[\tilde \phi]\int_{\tilde \phi(\Omega)}2s|x|^{2s-2} x\cdot\omega_x (\nabla_y v_1)(\nabla_y v_2)^t\,dz,
\end{align*}
and 
\begin{align*}
\lambda^{-1}[\tilde \phi]\int_{\partial \tilde \phi(\Omega)}& (\omega_y \mathbf{n}_y^t)(\nabla_x v_1)(\nabla_x v_2)^t\,d\sigma_z \\
=&\; \lambda^{-1}[\tilde \phi]\int_{\tilde \phi(\Omega)}\mathrm{div}_y \big((\nabla_x v_1)(\nabla_x v_2)^t\omega_y\big)\,dz\\
=& \;\lambda^{-1}[\tilde \phi]\int_{\tilde \phi(\Omega)}(\nabla_x v_1)(\nabla_x v_2)^t\mathrm{div}_y \omega_y\,dz 
+ \lambda^{-1}[\tilde \phi]\int_{\tilde \phi(\Omega)}\nabla_y \Big((\nabla_x v_1)(\nabla_x v_2)^t\Big)\omega_y^t\,dz \\
= & \;\lambda^{-1}[\tilde \phi]\int_{\tilde \phi(\Omega)}(\nabla_x v_1)(\nabla_x v_2)^t\mathrm{div}_y \omega_y\,dz
+ \lambda^{-1}[\tilde \phi]\int_{\tilde \phi(\Omega)}\nabla_x v_1\Big(D_y \nabla_x v_2\Big)\omega^t_y\,dz\\
&\,+ \lambda^{-1}[\tilde \phi]\int_{\tilde \phi(\Omega)}\nabla_x v_2\Big(D_y \nabla_x v_1\Big)\omega^t_y\,dz.
\end{align*}
Therefore, we deduce the following expression for $I_2$
\begin{equation}\label{eq:I2}
\begin{split}
I_2=&- \lambda^{-1}[\tilde \phi]\int_{\partial\tilde \phi(\Omega)} (\omega \mathbf{n}^t)(\nabla_G v_1)(\nabla_G v_2)^t\,d\sigma_z\\
& +\lambda^{-1}[\tilde \phi]\int_{\tilde \phi(\Omega)}|x|^{2s}(\nabla_y v_1)(\nabla_y v_2)^t\mathrm{div}_x \omega_x\,dz
+ \lambda^{-1}[\tilde \phi]\int_{\tilde \phi(\Omega)}|x|^{2s}\nabla_y v_1\Big(D_x \nabla_y v_2\Big)\omega^t_x\,dz\\
&+ \lambda^{-1}[\tilde \phi]\int_{\tilde \phi(\Omega)}|x|^{2s}\nabla_y v_2\Big(D_x \nabla_y v_1\Big)\omega^t_x\,dz+ \lambda^{-1}[\tilde \phi]\int_{\tilde \phi(\Omega)}2s|x|^{2s-2}x\omega_x^t(\nabla_y v_1)(\nabla_y v_2)^t\,dz\\
&+\lambda^{-1}[\tilde \phi]\int_{\tilde \phi(\Omega)}(\nabla_x v_1)(\nabla_x v_2)^t\mathrm{div}_y \omega_y\,dz
+ \lambda^{-1}[\tilde \phi]\int_{\tilde \phi(\Omega)}\nabla_x v_1\Big(D_y \nabla_x v_2\Big)\omega^t_y\,dz\\
&+ \lambda^{-1}[\tilde \phi]\int_{\tilde \phi(\Omega)}\nabla_x v_2\Big(D_y \nabla_x v_1\Big)\omega^t_y\,dz.
\end{split}
\end{equation}
Next, we turn to consider  $I_3$.  
We recall that  $\mathbf{n}_G$ is the outward field to $\partial\Omega$ defined in \eqref{grushinnormal}. Therefore
\begin{align}\label{hadform3}
I_3=&\,\lambda^{-1}[\tilde \phi]\int_{\partial\tilde \phi(\Omega)}\left(
  (\omega_x\nabla_xv_1^t)(\mathbf{n}_x\nabla_xv_2^t)+
  (\omega_x\nabla_xv_2^t)(\mathbf{n}_x\nabla_xv_1^t)\right) \,d\sigma_z\\\nonumber 
  &+ \lambda^{-1}[\tilde \phi]\int_{\partial\tilde \phi(\Omega)}\left(|x|^{2s} 
  (\omega_y\nabla_yv_1^t)(\mathbf{n}_y\nabla_yv_2^t)+|x|^{2s} 
  (\omega_y\nabla_yv_2^t)(\mathbf{n}_y\nabla_yv_1^t)\right) \,d\sigma_z\\ \nonumber
  =&\;\lambda^{-1}[\tilde \phi]\int_{\partial\tilde \phi(\Omega)}\left(
  (\omega_x\nabla_xv_1^t)(\mathbf{n}_G\nabla_Gv_2^t)+
  (\omega_x\nabla_xv_2^t)(\mathbf{n}_G\nabla_Gv_1^t)\right) \,d\sigma_z\\\nonumber
  &- \lambda^{-1}[\tilde \phi]\int_{\partial\tilde \phi(\Omega)} 
  |x|^{2s}(\omega_x\nabla_xv_1^t)(\mathbf{n}_y\nabla_yv_2^t)\,d\sigma_z 
  - \lambda^{-1}[\tilde \phi]\int_{\partial\tilde \phi(\Omega)} 
  |x|^{2s}(\omega_x\nabla_xv_2^t)(\mathbf{n}_y\nabla_yv_1^t)\,d\sigma_z\\\nonumber
  &+ \lambda^{-1}[\tilde \phi]\int_{\partial\tilde \phi(\Omega)}\left( 
  (\omega_y\nabla_yv_1^t)(\mathbf{n}_G\nabla_Gv_2^t)+ 
  (\omega_y\nabla_yv_2^t)(\mathbf{n}_G\nabla_Gv_1^t)\right) \,d\sigma_z\\ \nonumber
  &-  \lambda^{-1}[\tilde \phi]\int_{\partial\tilde \phi(\Omega)}
  (\omega_y\nabla_yv_1^t)(\mathbf{n}_x\nabla_xv_2^t)\,d\sigma_z
  -  \lambda^{-1}[\tilde \phi]\int_{\partial\tilde \phi(\Omega)}
  (\omega_y\nabla_yv_2^t)(\mathbf{n}_x\nabla_xv_1^t)\,d\sigma_z \\\nonumber
  =&\; \lambda^{-1}[\tilde \phi]\int_{\partial\tilde \phi(\Omega)} 
  (\omega \nabla v_1^t)(\mathbf{n}_G\nabla_Gv_2^t)+ (\omega \nabla v_2^t)(\mathbf{n}_G\nabla_Gv_1^t)\,d\sigma_z \\ \nonumber 
  & - \lambda^{-1}[\tilde \phi]\int_{\partial\tilde \phi(\Omega)} 
  |x|^{2s}(\omega_x\nabla_xv_1^t)(\mathbf{n}_y\nabla_yv_2^t)\,d\sigma_z 
  - \lambda^{-1}[\tilde \phi]\int_{\partial\tilde \phi(\Omega)} 
  |x|^{2s}(\omega_x\nabla_xv_2^t)(\mathbf{n}_y\nabla_yv_1^t)\,d\sigma_z\\\nonumber
   &-  \lambda^{-1}[\tilde \phi]\int_{\partial\tilde \phi(\Omega)}
  (\omega_y\nabla_yv_1^t)(\mathbf{n}_x\nabla_xv_2^t)\,d\sigma_z
  -  \lambda^{-1}[\tilde \phi]\int_{\partial\tilde \phi(\Omega)}
  (\omega_y\nabla_yv_2^t)(\mathbf{n}_x\nabla_xv_1^t)\,d\sigma_z.  
\end{align}
We consider the last four  terms in the right hand side of the previous equation:
\begin{align*}
 - \lambda^{-1}[\tilde \phi]\int_{\partial\tilde \phi(\Omega)} &
  |x|^{2s}(\omega_x\nabla_xv_1^t)(\mathbf{n}_y\nabla_yv_2^t)\,d\sigma_z 
  - \lambda^{-1}[\tilde \phi]\int_{\partial\tilde \phi(\Omega)} 
  |x|^{2s}(\omega_x\nabla_xv_2^t)(\mathbf{n}_y\nabla_yv_1^t)\,d\sigma_z\\\nonumber
  =& - \lambda^{-1}[\tilde \phi]\int_{\tilde \phi(\Omega)} \mathrm{div}_y\Big(|x|^{2s}(\omega_x\nabla_x v_1^t)\nabla_y v_2+|x|^{2s}(\omega_x\nabla_x v_2^t)\nabla_y v_1  \Big) \,dz\\
  =&-\lambda^{-1}[\tilde \phi]\int_{\tilde \phi(\Omega)} |x|^{2s}\Delta_y v_2 (\omega_x\nabla_x v_1^t) \,dz -\lambda^{-1}[\tilde \phi] \int_{\tilde \phi(\Omega)}|x|^{2s}\nabla_y(\omega_x \nabla_x v_1^t) \nabla_y v_2^t\,dz\\
  &-\lambda^{-1}[\tilde \phi]\int_{\tilde \phi(\Omega)} |x|^{2s}\Delta_y v_1 (\omega_x\nabla_x v_2^t) \,dz -\lambda^{-1}[\tilde \phi] \int_{\tilde \phi(\Omega)}|x|^{2s}\nabla_y(\omega_x \nabla_x v_2^t) \nabla_y v_1^t\,dz\\
  =&-\lambda^{-1}[\tilde \phi]\int_{\tilde \phi(\Omega)} |x|^{2s}\Delta_y v_2 (\omega_x\nabla_x v_1^t) \,dz -\lambda^{-1}[\tilde \phi]\int_{\tilde \phi(\Omega)} |x|^{2s}\Delta_y v_1 (\omega_x\nabla_x v_2^t) \,dz\\ 
    &-\lambda^{-1}[\tilde \phi] \int_{\tilde \phi(\Omega)}|x|^{2s}\omega_x (D_y\nabla_x v_1) \nabla_y v_2^t\,dz
    -\lambda^{-1}[\tilde \phi] \int_{\tilde \phi(\Omega)}|x|^{2s}\omega_x (D_y\nabla_x v_2) \nabla_y v_1^t\,dz\\
   &-\lambda^{-1}[\tilde \phi] \int_{\tilde \phi(\Omega)}|x|^{2s}\nabla_x v_1(D_y\omega_x ) \nabla_y v_2^t\,dz
   -\lambda^{-1}[\tilde \phi] \int_{\tilde \phi(\Omega)}|x|^{2s}\nabla_x v_2(D_y\omega_x ) \nabla_y v_1^t\,dz,
\end{align*} 
and similarly
\begin{align*}
-  \lambda^{-1}[\tilde \phi]\int_{\partial\tilde \phi(\Omega)}&
  (\omega_y\nabla_yv_1^t)(\mathbf{n}_x\nabla_xv_2^t)\,d\sigma_z
  -  \lambda^{-1}[\tilde \phi]\int_{\partial\tilde \phi(\Omega)}
  (\omega_y\nabla_yv_2^t)(\mathbf{n}_x\nabla_xv_1^t)\,d\sigma_z\\
   =&-\lambda^{-1}[\tilde \phi]\int_{\tilde \phi(\Omega)} \Delta_x v_2 (\omega_y\nabla_y v_1^t) \,dz -\lambda^{-1}[\tilde \phi]\int_{\tilde \phi(\Omega)} \Delta_x v_1 (\omega_y\nabla_y v_2^t) \,dz\\ 
    &-\lambda^{-1}[\tilde \phi] \int_{\tilde \phi(\Omega)}\omega_y (D_x\nabla_y v_1) \nabla_x v_2^t\,dz
    -\lambda^{-1}[\tilde \phi] \int_{\tilde \phi(\Omega)}\omega_y (D_x\nabla_y v_2) \nabla_x v_1^t\,dz\\
   &-\lambda^{-1}[\tilde \phi] \int_{\tilde \phi(\Omega)}\nabla_y v_1(D_x\omega_y ) \nabla_x v_2^t\,dz
   -\lambda^{-1}[\tilde \phi] \int_{\tilde \phi(\Omega)}\nabla_y v_2(D_x\omega_y ) \nabla_x v_1^t\,dz.
\end{align*}
We can now rewrite the right hand side of equation \eqref{hadform3} and deduce that
\begin{align}\label{eq:I3}
I_3=&\, \lambda^{-1}[\tilde \phi]\int_{\partial\tilde \phi(\Omega)} 
  (\omega \nabla v_1^t)(\mathbf{n}_G\nabla_Gv_2^t)+ (\omega \nabla v_2^t)(\mathbf{n}_G\nabla_Gv_1^t)\,d\sigma_z \\ \nonumber 
 &-\lambda^{-1}[\tilde \phi]\int_{\tilde \phi(\Omega)} |x|^{2s}\Delta_y v_2 (\omega_x\nabla_x v_1^t) \,dz -\lambda^{-1}[\tilde \phi]\int_{\tilde \phi(\Omega)} |x|^{2s}\Delta_y v_1 (\omega_x\nabla_x v_2^t) \,dz\\ \nonumber
    &-\lambda^{-1}[\tilde \phi] \int_{\tilde \phi(\Omega)}|x|^{2s}\omega_x (D_y\nabla_x v_1) \nabla_y v_2^t\,dz
    -\lambda^{-1}[\tilde \phi] \int_{\tilde \phi(\Omega)}|x|^{2s}\omega_x (D_y\nabla_x v_2) \nabla_y v_1^t\,dz\\\nonumber
   &-\lambda^{-1}[\tilde \phi] \int_{\tilde \phi(\Omega)}|x|^{2s}\nabla_x v_1(D_y\omega_x ) \nabla_y v_2^t\,dz
   -\lambda^{-1}[\tilde \phi] \int_{\tilde \phi(\Omega)}|x|^{2s}\nabla_x v_2(D_y\omega_x ) \nabla_y v_1^t\,dz\\\nonumber
   &-\lambda^{-1}[\tilde \phi]\int_{\tilde \phi(\Omega)} \Delta_x v_2 (\omega_y\nabla_y v_1^t) \,dz -\lambda^{-1}[\tilde \phi]\int_{\tilde \phi(\Omega)} \Delta_x v_1 (\omega_y\nabla_y v_2^t) \,dz\\ \nonumber
    &-\lambda^{-1}[\tilde \phi] \int_{\tilde \phi(\Omega)}\omega_y (D_x\nabla_y v_1) \nabla_x v_2^t\,dz
    -\lambda^{-1}[\tilde \phi] \int_{\tilde \phi(\Omega)}\omega_y (D_x\nabla_y v_2) \nabla_x v_1^t\,dz\\ \nonumber
   &-\lambda^{-1}[\tilde \phi] \int_{\tilde \phi(\Omega)}\nabla_y v_1(D_x\omega_y ) \nabla_x v_2^t\,dz
   -\lambda^{-1}[\tilde \phi] \int_{\tilde \phi(\Omega)}\nabla_y v_2(D_x\omega_y ) \nabla_x v_1^t\,dz.
\end{align}
By equalities \eqref{hadform2}, \eqref{hadform2bis}, \eqref{eq:I2}, and \eqref{eq:I3} and by noting that $(D_x\nabla_y v_l)^t =D_y\nabla_x v_l$, we deduce that
 \begin{align*}
 Q_{G,\tilde \phi}\left[  d_{|\phi = \tilde \phi} T_{G,\phi}[\psi][u_1],u_2\right] = &
   - \lambda^{-1}[\tilde \phi]\int_{\partial\tilde \phi(\Omega)} (\omega \mathbf{n}^t)(\nabla_G v_1)(\nabla_G v_2)^t\,d\sigma_z\\
&+ \lambda^{-1}[\tilde \phi]\int_{\partial\tilde \phi(\Omega)} 
  (\omega \nabla v_1^t)(\mathbf{n}_G\nabla_Gv_2^t)+ (\omega \nabla v_2^t)(\mathbf{n}_G\nabla_Gv_1^t)\,d\sigma_z.
 \end{align*}
We note that,  since $v_1,v_2 \in W^{1,2}_0(\tilde \phi(\Omega))\cap W^{2,2}(\tilde \phi(\Omega))$, the gradients $\nabla v_1, \nabla v_2 $ are parallel to $\mathbf{n}$ on $\partial \tilde\phi( \Omega)$ and $\nabla_Gv_1, \nabla_Gv_2$ are parallel to $\mathbf{n}_G$ on $\partial \tilde\phi( \Omega)$. Accordingly,
\begin{align*}
\nabla_G v_i = \nabla v_i\, I_G= (\nabla v_i \, \mathbf{n}^t)\mathbf{n} \,I_G = 
  \frac{\partial v_i }{\partial \mathbf{n}} \,\mathbf{n}_G, \qquad \mbox{ a.e. on  } \partial \tilde\phi( \Omega) \quad \mbox{ for } i \in \{1,2\}.
\end{align*}
 Then 
\begin{align*}
  \int_{\partial\tilde \phi(\Omega)} (\omega \mathbf{n}^t)(\nabla_G v_1)(\nabla_G v_2)^t\,d\sigma_z
     = \int_{\partial\tilde \phi(\Omega)} ( \omega \mathbf{n}^t) \frac{\partial v_1}{\partial \mathbf{n}}\frac{\partial v_2}{\partial \mathbf{n}}|\mathbf{n}_G|^2\,d\sigma_z\, ,
\end{align*}
and 
\begin{align*}
\int_{\partial\tilde \phi(\Omega)} 
  (\omega \nabla v_1^t)(\mathbf{n}_G\nabla_Gv_2^t)+ (\omega \nabla v_2^t)(\mathbf{n}_G\nabla_Gv_1^t)\,d\sigma_z
 =2\int_{\partial\tilde \phi(\Omega)} ( \omega \mathbf{n}^t) \frac{\partial v_1}{\partial \mathbf{n}}\frac{\partial v_2}{\partial \mathbf{n}}|\mathbf{n}_G|^2\,d\sigma_z.
\end{align*}
 We can finally conclude that
\begin{align*}
 Q_{G,\tilde \phi}\left[  d_{|\phi = \tilde \phi} T_{G,\phi}[\psi][u_1],u_2\right] =  \lambda^{-1}[\tilde \phi]\int_{\partial\tilde \phi(\Omega)} ( \omega \mathbf{n}^t) \frac{\partial v_1}{\partial \mathbf{n}}\frac{\partial v_2}{\partial \mathbf{n}}|\mathbf{n}_G|^2\,d\sigma_z.
  \end{align*}  
 \end{proof}
  Now, combining Theorem \ref{hadf}, formula \eqref{hadf2},  and Lemma \ref{Qformula} we are able to deduce our main result regarding the Hadamard-type formula for the shape differential of the symmetric functions of the eigenvalues.
 \begin{theorem}\label{Nhadf}
 Let  $\Omega$ and $O$ be as in \eqref{Omega_def}.   Let $F$ be a finite nonempty subset 
 of $\mathbb{N}$. Let $\tau \in \{1,\ldots,|F|\}$.  
Let $\tilde \phi \in \Theta^F_{\Omega,O}$ and  let
 $\lambda_F[\tilde \phi]$ be the common value of all the eigenvalues $\{\lambda_j[\tilde \phi]\}_{j \in F}$. 
 Let $\{v_l\}_{l \in F}$ be an orthonormal basis in 
 $(W^{1,2}_{G,0}(\tilde \phi(\Omega)),Q_{G})$
  of the eigenspace associated with $\lambda_F[\tilde \phi]$.
 Suppose that  $\tilde \phi(\Omega)$ is of class $C^1$ and $\{v_l\}_{l \in F} \subseteq W^{1,2}_0(\tilde \phi(\Omega)) \cap W^{2,2}(\tilde \phi(\Omega))$. Then  the Frech\'et differential of the map  $\Lambda_{F,\tau}$ at the
 point $\tilde \phi$ is delivered by the formula
 \begin{align}\label{Nhadf1}
 d_{|\phi = \tilde \phi}(\Lambda_{F,\tau})[\psi]= \, & -\lambda_F^{\tau}[\tilde \phi]\binom{|F|-1}{\tau-1}  \sum_{l \in F} 
     \int_{\partial\tilde \phi(\Omega)} (\psi \circ \tilde \phi^{(-1)} \mathbf{n}^t) \left(\frac{\partial v_l}{\partial \mathbf{n}}\right)^2|\mathbf{n}_G|^2\,d\sigma  \quad \forall \psi \in L_{\Omega,O}.
  \end{align}
 \end{theorem}
  \begin{remark}\label{rem:had}
 In order to prove   formula \eqref{Qformula1} and subsequently the Hadamard formula \eqref{Nhadf1} we had to assume some extra regularity for the eigenfunctions to avoid regularity problems near 
$\partial\tilde \phi(\Omega) \cap \{x=0\}$. However, if $\psi \in  L_{\Omega,O}$ is such that 
$\psi_{|\Omega \cap O}=0$,
then, since any problem around the degenerate set is canceled by $\psi$, formulas \eqref{Qformula1} and 
 \eqref{Nhadf1}   hold without requiring that the eigenfunctions are of class $W^{1,2}_0 \cap W^{2,2}$. 
 \end{remark}
Next, we consider the case of a family of domain perturbations depending real analytically upon one scalar parameter. In this case it is possible to prove a Rellich-Nagy-type theorem and describe all 
 the branches splitting from a multiple eigenvalue of multiplicity $m$ by means of $m$ real analytic functions of the scalar parameter. Namely, we have the following.
 \begin{theorem}\label{RN}
Let  $\Omega$ and $O$ be as in \eqref{Omega_def}.
Let $\tilde \phi \in \mathcal{A}_{\Omega,O}$ and $\{\phi_\varepsilon\}_{\varepsilon \in \mathbb{R}}\subseteq \mathcal{A}_{\Omega,O}$ 
be a family depending real analytically on $\varepsilon$ and such that 
$\phi_0 = \tilde \phi$.   Let $\lambda$ be a Dirichlet Grushin eigenvalue on $\tilde \phi (\Omega)$ of multiplicity $m$. Let $\lambda = 
\lambda_n[\tilde\phi] = \cdots = \lambda_{n+m-1}[\tilde \phi]$ for some $n \in \mathbb{N}$.
 Let $v_1, \ldots, v_m$ be an orthonormal basis  in 
 $(W^{1,2}_{G,0}(\tilde \phi(\Omega)),Q_{G})$   of the eigenspace 
 associated with $\lambda$. Suppose that $\tilde \phi(\Omega)$ is of class $C^1$ and that 
  $v_1, \ldots, v_m \in W^{1,2}_{0}(\tilde \phi(\Omega)) \cap W^{2,2}(\tilde \phi(\Omega))$.
 
Then there exist an open interval $I \subseteq \mathbb{R}$ containing zero  and $m$ real analytic functions 
$g_1, \ldots , g_m$ from $I$ to $\mathbb{R}$ such that 
$\{\lambda_n[\phi_\varepsilon], \ldots , \lambda_{n+m-1}[ \phi_\varepsilon]\} = \{g_1(\varepsilon), \ldots , g_m(\varepsilon)\}$ for all $\varepsilon \in I$. Moreover, the derivatives $g_1'(0),\ldots,g_m'(0)$ of the functions
$g_1, \ldots , g_m$ at zero coincides with the eigenvalues of the matrix
\begin{equation}\label{RN1}
\left( -\lambda\int_{\partial\tilde \phi(\Omega)}  (\dot\phi_0 \circ  {\tilde\phi}^{(-1)} \mathbf{n}^t)  \frac{\partial v_i}{\partial \mathbf{n}}\frac{\partial v_j}{\partial \mathbf{n}}|\mathbf{n}_G|^2\,d\sigma\right)_{i,j= 1,\ldots, m},
\end{equation}
 where $\dot \phi_0$ denotes the derivative at $\varepsilon = 0$ of the map 
$\varepsilon \mapsto \phi_\varepsilon$. The same conclusion holds dropping the regularity assumption on the eigenfunctions and requiring that ${\phi_\varepsilon}_{|\Omega \cap O}$ is the identity map for all $\varepsilon \in \mathbb{R}$.
\end{theorem}
\begin{proof}
The proof follows by the abstract Rellich-Nagy-type theorem of 
Lamberti and Lanza di Cristoforis \cite[Theorem 2.27, Corollary 2.28]{LaLa04} applied
to the family of operators  $(T_{G,\phi_\varepsilon})_{\varepsilon \in \mathbb{R}}$ defined in \eqref{def:Tphi}, which guarantees that 
there exist an open interval $I$ containing zero and $m$ real analytic functions $f_1, \ldots, f_m$ 
from $I$ to $\mathbb{R}$  such that $\{\lambda_n[\tilde\phi]^{-1} , \ldots , \lambda_{n+m-1}[\tilde \phi]^{-1}\} = 
\{f_1(\varepsilon), \ldots, f_m(\varepsilon)\}$  for all $\varepsilon \in I$ and that, if we set 
$u_i = v_i \circ \tilde \phi$ for all $i =1,\ldots,m$,  the set $\{f_1'(0), \ldots, f_m'(0)\}$ coincides with the set of eigenvalues of the matrix
\[
 \left(Q_{G,\tilde \phi}\Big[d_{|\phi = \tilde \phi} T_{G,\phi}[\dot\phi_0][u_i],u_j\Big] \right)_{i,j=1,\ldots,m}.
\]
Then we can conclude by setting $g_i = f_i^{-1}$ for all $i=1\ldots,m$ and exploiting Lemma \ref{Qformula}.
The last part of the statement follows by the same arguments together with Remark \ref{rem:had}.
\end{proof}
We conclude this section with the following remark on the scalar product used.
\begin{remark}\label{rem:normaliz}
The above formulas are obtained assuming that 
 the orthogonality of the eigenfunctions
 is taken in $(W^{1,2}_{G,0}(\tilde \phi(\Omega)),Q_{G})$. If instead one prefers to consider $\{v_l\}_{l \in F}$ to be an orthonormal basis in 
 $L^2(\tilde \phi(\Omega))$ endowed with its standard scalar product, then formula \eqref{Nhadf1} of Theorem \ref{Nhadf} can be rewritten as 
  \begin{align*}
 d_{|\phi = \tilde \phi}(\Lambda_{F,\tau})[\psi]= \, & -\lambda_F^{\tau-1}[\tilde \phi]\binom{|F|-1}{\tau-1}  \sum_{l \in F} 
     \int_{\partial\tilde \phi(\Omega)} (\psi \circ \tilde \phi^{(-1)} \mathbf{n}^t) \left(\frac{\partial v_l}{\partial \mathbf{n}}\right)^2|\mathbf{n}_G|^2\,d\sigma  \quad \forall \psi \in L_{\Omega,O}.
  \end{align*}
  Similarly, in Theorem \ref{RN} we can choose $v_1, \ldots, v_m$ to be an orthonormal basis  in 
  $L^2(\tilde \phi(\Omega))$ and in this case the matrix \eqref{RN1} becomes 
\begin{equation}\label{RN1B}
\left( -\int_{\partial\tilde \phi(\Omega)}  (\dot\phi_0 \circ  {\tilde\phi}^{(-1)} \mathbf{n}^t)  \frac{\partial v_i}{\partial \mathbf{n}}\frac{\partial v_j}{\partial \mathbf{n}}|\mathbf{n}_G|^2\,d\sigma\right)_{i,j= 1,\ldots, m}.
\end{equation}
\end{remark}

\section{Critical shapes and overdetermined problems}\label{sec:critical}
In this section we consider the problem of studying the critical shapes for the symmetric functions of the eigenvalues 
under isovolumetric  and isoperimetric perturbations. Let $\Omega$ be a bounded open subset of $\mathbb{R}^N$. 
Let $F$ be a finite nonempty subset 
 of $\mathbb{N}$.  We set 
 \begin{align*}
&\mathcal{A}^F_\Omega := \big\{ \phi \in \mathrm{Lip}(\Omega)^N : \exists\, O \mbox{ as in } \eqref{Omega_def}
 \mbox{ such that } \phi \in \mathcal{A}^F_{\Omega,O}\big\},\\
&\Theta^F_\Omega := \big\{\phi \in  \mathcal{A}^F_{\Omega} :  \lambda_n[\phi] = \lambda_m[\phi], \,
 \forall n,m \in F\big\}.
\end{align*}
 If $\phi \in \mathcal{A}^F_\Omega$, we denote by $\mathcal{V}[\phi]$ and  $\mathcal{P}[\phi]$  the volume and the perimeter of the set $\phi(\Omega)$, respectively.
Let $\tau \in \{1,\ldots,|F|\}$. 
 Our interest 
in critical shapes mainly comes from  the study of optimization problems of the following type:
\begin{equation}\label{shapeoptv}
\max_{\mathcal{V}[\phi]=\mathrm{const.}} \Lambda_{F,\tau}[\phi] \quad \mbox{ or } \quad 
\min_{\mathcal{V}[\phi]=\mathrm{const.}} \Lambda_{F,\tau}[\phi],
\end{equation}
as well as 
\begin{equation}\label{shapeoptp}
\max_{\mathcal{P}[\phi]=\mathrm{const.}} \Lambda_{F,\tau}[\phi] \quad \mbox{ or } \quad 
\min_{\mathcal{P}[\phi]=\mathrm{const.}} \Lambda_{F,\tau}[\phi].
\end{equation}
Indeed, a first step towards the understanding of problems \eqref{shapeoptv} and \eqref{shapeoptp} is to find the critical shapes 
under volume and perimeter constraints, respectively.
\subsection{The isovolumetric problem}
First, we consider the problem of finding critical shapes under isovolumetric perturbations. 
Let $\Omega$ and $O$   be as in \eqref{Omega_def}.
The volume functional $\mathcal{V}[\cdot]$ is the map from $\mathcal{A}^F_{\Omega}$ to $\mathbb{R}$ 
defined by
\[
\mathcal{V}[\phi] := \int_{\phi(\Omega)} 1\,dz= \int_{\Omega} |\det D\phi|\,dz \qquad \forall \phi \in \mathcal{A}_{\Omega}^F.
\]
It is easily seen that $\mathcal{V}_{|\mathcal{A}^F_{\Omega,O}}$ is real analytic and that, by standard rules of calculus in Banach spaces, 
its differential at the point $\tilde \phi \in \mathcal{A}^F_{\Omega,O}$ is delivered by 
\begin{equation*}
d_{|\phi = \tilde \phi}\mathcal{V}_{|\mathcal{A}^F_{\Omega,O}}[\psi] = \int_{\tilde \phi(\Omega)}\mathrm{div}\left(\psi \circ \tilde\phi^{(-1)}\right) \,dz \qquad 
\forall \psi \in L_{\Omega,O}.
\end{equation*}
Under the assumption that $ \tilde \phi(\Omega)$ is of class $C^1$,  the above differential 
can be rewritten as 
\begin{equation}\label{diffV1}
d_{|\phi = \tilde \phi}\mathcal{V}_{|\mathcal{A}^F_{\Omega,O}}[\psi] 
= \int_{\partial \tilde \phi(\Omega)} \left(\psi \circ \tilde\phi^{(-1)}\right) \cdot \mathbf{n} \,d\sigma \qquad 
\forall \psi \in L_{\Omega,O}.
\end{equation}
For $M \in \mathopen ]0,+\infty[$  we set 
\begin{align*}
V[M] := \left\{\phi \in\mathcal{A}^F_{\Omega} : \mathcal{V}[\phi] = M  \right\}.
\end{align*}
Suppose now that a shape $\tilde \phi \in \mathcal{A}^F_{\Omega,O}$ is a maximizer  (or a minimizer) in the shape optimization problems \eqref{shapeoptv}  under the volume constraint 
$\phi \in V[M]$ among all the shapes $\phi$ in $\mathcal{A}^F_{\Omega}$. 
 Then,  for all open sets $O' \subseteq \mathbb{R}^N$ such that
 \begin{equation}\label{osubset}
  O' \subseteq O  \,\mbox{  and  }\overline{\Omega} \cap \{x=0\} \subseteq O',
 \end{equation}
$\tilde \phi$  is a maximizer  (or a minimizer)  under the volume constraint 
$\phi \in V[M]$ for all the shapes $\phi$ in $\mathcal{A}^F_{\Omega,O'}$.
 Accordingly,  for all $O' \subseteq O$ such that $\overline{\Omega} \cap \{x=0\}   \subseteq O'$, $\tilde \phi$ is a critical point for 
 ${\Lambda_{F,\tau}}_{|\mathcal{A}^F_{\Omega,O'}}$ under the volume constraint $\phi \in V[M]$, in other words:
\[
\mathrm{ker} \,d_{\phi = \tilde \phi} \mathcal{V}_{|\mathcal{A}^F_{\Omega,O'}} \subseteq \mathrm{ker}\, 
d_{\phi = \tilde \phi}{\Lambda_{F,\tau}}_{|\mathcal{A}^F_{\Omega,O'}} \qquad \forall 
 O'  \subseteq \mathbb{R}^N \mbox{ as in } \eqref{osubset}.
\]
By the Lagrange multipliers theorem, the above condition is equivalent to the fact that for all open sets
 $O'  \subseteq \mathbb{R}^N$ as in  \eqref{osubset}, there exists a constant 
$c_{O'} \in \mathbb{R}$ (a Lagrange multiplier) such that  
\begin{equation}\label{defcriticalv}
 d_{\phi = \tilde \phi} {\Lambda_{F,\tau}}_{|\mathcal{A}^F_{\Omega,O'}}[\psi] 
 + c_{O'}d_{\phi = \tilde \phi} \mathcal{V}_{|\mathcal{A}^F_{\Omega,O'}}[\psi]  = 0
\qquad \forall \psi \in L_{\Omega, O'}\, .
\end{equation}
Inspired by the above discussion, we introduce the following definition.
\begin{definition}\label{defcritv}
Let $\Omega$ be a bounded open subset of $\mathbb{R}^N$.  Let $M \in \mathopen]0, +\infty[$. 
Let $F$ be a finite nonempty subset of $\mathbb{N}$.  Let $\tau \in \{1,\ldots,|F|\}$.  Let $\tilde \phi \in \mathcal{A}^F_\Omega \cap V[M]$. We say that $\tilde \phi$ is critical for $\Lambda_{F,\tau}$ under the
volume constraint $\phi \in V[M]$ if there exists
a bounded open subset $O$ of $\mathbb{R}^N$ with $\overline{\Omega} \cap \{x=0\} \subseteq O$ such that $\tilde \phi \in \mathcal{A}^F_{\Omega,O}$ and such that for all open sets $O'  \subseteq \mathbb{R}^N$ as in \eqref{osubset} there exists $c_{O'} \in \mathbb{R}$  such that \eqref{defcriticalv} holds.
\end{definition}
In the  following proposition we prove a necessary condition for the criticality of 
shapes under isovolumetric perturbations.
\begin{theorem}\label{criticcond}
Let $\Omega$ be a bounded open subset of $\mathbb{R}^N$.
 Let $F$ be a finite nonempty subset 
 of $\mathbb{N}$ and $\tau \in \{1,\ldots,|F|\}$. 
 Let $M \in \mathopen ]0,+\infty[$.   Let $\tilde \phi \in \Theta^F_{\Omega} \cap V[M]$ and  let
 $\lambda_F[\tilde \phi]$ be the common value of all the eigenvalues $\{\lambda_j[\tilde \phi]\}_{j \in F}$. 
 Assume that  $\tilde \phi(\Omega)$ is of class $C^1$.
 Let $\{v_l\}_{l \in F}$ be an orthonormal basis in 
$(W^{1,2}_{G,0}(\tilde \phi(\Omega)),Q_{G})$
  of the eigenspace associated with $\lambda_F[\tilde \phi]$.
 If  $\tilde \phi$ is a critical shape for $\Lambda_{F,\tau}$  under the volume constraint  $\phi \in V[M]$,
then there exists a constant $c_1 \in \mathbb{R}$ such that 
\begin{equation}\label{criticcond1}
\sum_{l \in F} \left(\frac{\partial v_l}{\partial \mathbf{n}}\right)^2|\mathbf{n}_G|^2 = c_1 \qquad \mbox{ a.e. on } \partial\tilde \phi(\Omega) \setminus \{x=0\}.
\end{equation}
\end{theorem}
\begin{proof}
Let $\tilde \phi$ be a critical shape for $\Lambda_{F,\tau}$ under the volume constraint 
$\tilde \phi \in V[M]$ and let  $O \subseteq \mathbb{R}^N$ be as in Definition \ref{defcritv}. For  $O' \subseteq \mathbb{R}^N$  as in \eqref{osubset} we set 
\[
 L_{\Omega,O'}^0 := \{ \psi \in L_{\Omega,O'}: \psi_{|\Omega \cap O'}=0\}.
\]
Then, by Theorem \ref{Nhadf} and Remark \ref{rem:had}, we have that
\begin{align}\label{criticcond2}
 d_{|\phi = \tilde \phi}(\Lambda_{F,\tau})[\psi]= \, & -\lambda_F^{\tau}[\tilde \phi]\binom{|F|-1}{\tau-1}  \sum_{l \in F} 
     \int_{\partial\tilde \phi(\Omega)} (\psi \circ \tilde \phi^{(-1)} \mathbf{n}^t) \left(\frac{\partial v_l}{\partial \mathbf{n}}\right)^2|\mathbf{n}_G|^2\,d\sigma \quad \forall \psi \in  L_{\Omega,O'}^0.
  \end{align}
Thus, formula \eqref{criticcond2}, formula \eqref{diffV1} for the differential of the volume functional, and Definition \ref{defcritv} imply that for  all  open sets
 $O' \subseteq \mathbb{R}^N$  as in \eqref{osubset} there exists a constant $c_{O'}\in \mathbb{R}$ such that 
 \begin{align*}
  \sum_{l \in F} 
     \int_{\partial\tilde \phi(\Omega)} \left(\psi \circ \tilde \phi^{(-1)}\right) \cdot \mathbf{n} \left(\frac{\partial v_l}{\partial \mathbf{n}}\right)^2|\mathbf{n}_G|^2\,d\sigma  + 
     c_{O'}  \int_{\partial \tilde \phi(\Omega)} &\left(\psi \circ \tilde \phi^{(-1)}\right) \cdot \mathbf{n} \,d\sigma=0
    \qquad \forall \psi \in  L_{\Omega,O'}^0\, .
     \end{align*}
  On the other hand, if  $O'  \subseteq \mathbb{R}^N$ is as in \eqref{osubset}, then 
  $L_{\Omega,O}^0 \subseteq L_{\Omega,O'}^0$.
  Hence,  $c_{O'}=c_{O}$. That is 
  \begin{align}\label{criticcond3}
  \sum_{l \in F} 
     \int_{\partial\tilde \phi(\Omega)} \left(\psi \circ \tilde \phi^{(-1)}\right) \cdot \mathbf{n} \left(\frac{\partial v_l}{\partial \mathbf{n}}\right)^2|\mathbf{n}_G|^2\,d\sigma  + 
     c_{O}  \int_{\partial \tilde \phi(\Omega)} &\left(\psi \circ \tilde \phi^{(-1)}\right) \cdot \mathbf{n} \,d\sigma=0
    \qquad \forall \psi \in L_{\Omega,O'}^0\, .
     \end{align}
 By the Fundamental Lemma of Calculus of Variations one can realize that \eqref{criticcond3} implies that there exists 
a constant $c_1 \in \mathbb{R}$ such that \eqref{criticcond1} holds.
\end{proof}
\begin{remark}\label{rem:criticcond}
If one assume that the eigenfunctions are of class $W^{1,2}_{0}(\tilde \phi(\Omega)) \cap W^{2,2}(\tilde \phi(\Omega))$, then it is easily seen that condition  \eqref{criticcond1} becomes also sufficient for the criticality of shapes under isovolumetric perturbations. Moreover, if the $(N-1)$-dimensional measure of $\partial\tilde \phi(\Omega) \cap \{x=0\}$ is zero, {\it i.e.} 
$\left|\partial\tilde \phi(\Omega) \cap \{x=0\}\right|_{N-1}=0$,
then \eqref{criticcond1} can be rewritten  as 
\begin{equation*}
\sum_{l \in F} \left(\frac{\partial v_l}{\partial \mathbf{n}}\right)^2|\mathbf{n}_G|^2 = c_1 \qquad \mbox{ a.e. on } \partial\tilde \phi(\Omega).
\end{equation*}
We note that $\left|\partial\tilde \phi(\Omega) \cap \{x=0\}\right|_{N-1}=0$ is always verified when 
$\mathrm{dim}\{x=0\} = k < N-1$.
\end{remark}
The previous theorem suggests considering the overdetermined system
\begin{equation}\label{overdetv}
\begin{cases}
-\Delta_G u_l = \lambda_j u_l \quad &\mbox{ in } \Omega, \,\forall l \in \{1,\ldots,m\},\\
u_l = 0 \quad &\mbox{ on } \partial \Omega,  \,\forall l \in \{1,\ldots,m\}, \\
\sum_{l =1}^m \left(\frac{\partial u_l}{\partial \mathbf{n}}\right)^2|\mathbf{n}_G|^2 =   \mbox{const.} \quad &\mbox{ on } \partial \Omega.
\end{cases}
\end{equation}
Here, $\lambda_j$ is the $j$-th eigenvalue which has multiplicity 
$m \in \mathbb{N}$ and  $u_1,\ldots, u_m$ 
is a corresponding orthonormal basis of eigenvalues in $W^{1,2}_{G,0}(\Omega)$ such that the last condition 
of system \eqref{overdetv} makes sense (for example $\{u_l\}_{l =1,\ldots,m} \subseteq W^{1,2}_0(\Omega) \cap 
W^{2,2}(\Omega)$ when $\Omega$ is of class $C^1$).
System \eqref{overdetv} is the Grushin analog 
of the well-known overdetermined system for the Laplace operator:
\begin{equation}\label{overdetLap}
\begin{cases}
-\Delta u_l = \lambda_j u_l \quad &\mbox{ in } \Omega, \,\forall l \in \{1,\ldots,m\},\\
u_l = 0 \quad &\mbox{ on } \partial \Omega,  \,\forall l \in \{1,\ldots,m\}, \\
\sum_{l =1}^m \left(\frac{\partial u_l}{\partial \mathbf{n}}\right)^2 = \mbox{const.} \quad &\mbox{ on } \partial \Omega.
\end{cases}
\end{equation}
It is known that  system \eqref{overdetLap} is satisfied when $\Omega$ is a ball (see Lamberti and Lanza 
de Cristoforis \cite{LaLa06}). Moreover, if $\Omega$ is connected and $j=1$ (and then $m=1$),    problem \eqref{overdetLap} is satisfied 
if and only if $\Omega$ is a ball (see Henry \cite{He82}).

  It would be of great interest characterizing those 
bounded domains such that system \eqref{overdetv} is
 satisfied or, at least, find some shapes for which  
it is satisfied. These problems, to the best of the authors' knowledge, are open.
\subsection{The isoperimetric problem}
Next, we switch to consider the isoperimetric problem. In this section we assume that $\Omega$ is a bounded open subset of $\mathbb{R}^N$ of class $C^{2}$. Let $O$  be as in \eqref{Omega_def}.  We set 
 \begin{align*}
&L^*_{\Omega,O} := L_{\Omega,O}  \cap C^{2}\left(\overline{\Omega}, \mathbb{R}^N\right)  
&&{\mathcal{A}^*}^F_{\Omega,O} := \mathcal{A}^F_{\Omega,O} \cap C^{2}\left(\overline{\Omega}, \mathbb{R}^N\right)\\
 &{\mathcal{A}^*}^F_\Omega:= \mathcal{A}^F_\Omega \cap C^{2}\left(\overline{\Omega}, \mathbb{R}^N\right) &&{\Theta^*}^F_\Omega :=  \Theta^F_\Omega \cap C^{2}\left(\overline{\Omega}, \mathbb{R}^N\right).
\end{align*}
The set $L^*_{\Omega,O}$ is a Banach subspace of
$C^{2}\left(\overline{\Omega}, \mathbb{R}^N\right)$ and  ${\mathcal{A}^*}^F_{\Omega,O}$
is open in $L^*_{\Omega,O}$.  The perimeter functional $\mathcal{P}[\cdot]$ is the map from ${\mathcal{A}^*}^F_{\Omega}$ to $\mathbb{R}$ 
defined by
\[
\mathcal{P}[\phi] := \int_{\partial\phi(\Omega)} 1\,d\sigma =
  \int_{\partial\Omega} \left|\mathbf{n} (D\phi)^{-1} \right| |\det D\phi|\,d\sigma  \qquad \forall \phi \in {\mathcal{A}^*}_{\Omega}^F.
\]
The map $\mathcal{P}_{|{\mathcal{A}^*}^F_{\Omega,O}}$ is real analytic and its differential at a point 
$\tilde \phi \in {\mathcal{A}^*}^F_{\Omega,O}$  is delivered by 
\begin{equation*}
d_{|\phi = \tilde \phi}\mathcal{P}_{|{\mathcal{A}^*}^F_{\Omega,O}}[\psi] = \int_{\partial \tilde \phi(\Omega)} 
\left(\psi \circ \tilde\phi^{(-1)}\right) \cdot \mathbf{n} \,\mathcal{H}\,d\sigma \qquad 
\forall \psi \in L^*_{\Omega,O},
\end{equation*}
where $\mathcal{H} = \mathrm{div}\, \mathbf{n}$ denotes the mean curvature of $\partial \tilde \phi(\Omega)$
(see \cite{La14}). For $M \in \mathopen ]0,+\infty[$  we set 
\begin{align*}
P[M] := \left\{\phi \in{\mathcal{A}^*}^F_{\Omega} : \mathcal{P}[\phi] = M  \right\}.
\end{align*}
Motivated by the isoperimetric optimization problems \eqref{shapeoptp}, we introduce the following definition. 
\begin{definition}
Let $\Omega$ be a bounded open subset of $\mathbb{R}^N$ of class $C^{2}$. 
Let $M \in \mathopen]0, +\infty[$. 
Let $F$ be a finite nonempty subset of $\mathbb{N}$.  Let $\tau \in \{1,\ldots,|F|\}$. 
Let $\tilde \phi \in {\mathcal{A}^*}^F_\Omega \cap P[M]$. We say that $\phi$ is critical for $\Lambda_{F,\tau}$ under the
 perimeter constraint $\phi \in P[M]$ if there exists
a bounded open subset $O$ of $\mathbb{R}^N$ with $\overline{\Omega} \cap \{x=0\} \subseteq O$ such that $\tilde \phi \in {\mathcal{A}^*}^F_{\Omega,O}$ and such that for all open sets $O'  \subseteq \mathbb{R}^N$ as in \eqref{osubset} there exists $c_{O'} \in \mathbb{R}$  such that 
\begin{equation*}
 d_{\phi = \tilde \phi} {\Lambda_{F,\tau}}_{|{\mathcal{A}^*}^F_{\Omega,O'}}[\psi] 
 + c_{O'}d_{\phi = \tilde \phi} \mathcal{P}_{|{\mathcal{A}^*}^F_{\Omega,O'}}[\psi]  = 0
\qquad \forall \psi \in L_{\Omega, O'}^*\,.
\end{equation*}
\end{definition}
Following the lines of the previous section, we are able to prove the following necessary condition for critical 
shapes under isoperimetric perturbations.
\begin{theorem}\label{criticcondp}
Let $\Omega$ be a bounded open subset of $\mathbb{R}^N$ of class $C^{2}$.
 Let $F$ be a finite nonempty subset 
 of $\mathbb{N}$ and $\tau \in \{1,\ldots,|F|\}$. 
 Let $M \in \mathopen ]0,+\infty[$.   Let $\tilde \phi \in {\Theta^*}^F_{\Omega} \cap P[M]$ and  let
 $\lambda_F[\tilde \phi]$ be the common value of all the eigenvalues $\{\lambda_j[\tilde \phi]\}_{j \in F}$. 
 Let $\{v_l\}_{l \in F}$ be an orthonormal basis in 
 $(W^{1,2}_{G,0}(\tilde \phi(\Omega)),Q_{G})$
  of the eigenspace associated with $\lambda_F[\tilde \phi]$. 
 If $\tilde \phi$ is a critical shape for $\Lambda_{F,\tau}$  under the perimeter constraint $\phi \in P[M]$, then
 there exists a constant $c_2 \in \mathbb{R}$ such that
\begin{equation}\label{criticcondp1}
\sum_{l \in F} \left(\frac{\partial v_l}{\partial \mathbf{n}}\right)^2|\mathbf{n}_G|^2= c_2\mathcal{H} \qquad \mbox{ a.e. on } \partial\tilde \phi(\Omega) \setminus \{x=0\}.
\end{equation}
\end{theorem}
\begin{remark}
As in Remark \ref{rem:criticcond}, if one assume that the eigenfunctions are of class $W^{1,2}_{0}(\tilde \phi(\Omega)) \cap W^{2,2}(\tilde \phi(\Omega))$, then condition  \eqref{criticcondp1} becomes also sufficient for the criticality of shapes under isoperimetric perturbations. 
Also, when $\left|\partial\tilde \phi(\Omega) \cap \{x=0\}\right|_{N-1}=0$, the same considerations as in Remark \ref{rem:criticcond} can be done.
\end{remark}
As before, Theorem \ref{criticcondp} suggests that it would be of interest to study the overdetermined system
\begin{equation*}
\begin{cases}
-\Delta_G u_l = \lambda_j u_l \quad &\mbox{ in } \Omega, \,\forall l \in \{1,\ldots,m\},\\
u_l = 0 \quad &\mbox{ on } \partial \Omega,  \,\forall l \in \{1,\ldots,m\}, \\
\sum_{l =1}^m \left(\frac{\partial u_l}{\partial \mathbf{n}}\right)^2|\mathbf{n}_G|^2 =  c_2 \mathcal{H}
 \quad &\mbox{ on } \partial \Omega,
\end{cases}
\end{equation*}
for some constant $c_2 \in \mathbb{R}$.

\section{The Rellich-Pohozaev identity for the Grushin Laplacian}\label{sec:rellich}

The aim of this section is to collect a new proof of the Rellich-Pohozaev identity for the Grushin Laplacian. 
 Let $\Omega$ be a bounded open subset of $\mathbb{R}^N$ of class 
$C^1$. Let 
$\lambda$ be an eigenvalue of the Dirichlet Grushin Laplacian, {\it i.e.} an eigenvalue of \eqref{wsp}. Let 
$u$ be an eigenfunction in $W^{1,2}_{G,0}(\Omega)$ corresponding to $\lambda$ normalized with $\|u\|_{L^2(\Omega)}=1$. Suppose that $u \in W^{1,2}_0(\Omega) \cap W^{2,2}(\Omega)$. Then 
the Rellich-Pohozaev identity reads 
\begin{equation}\label{eq:rp}
\lambda = \frac{1}{2}\int_{\partial\Omega}
\left(\frac{\partial u}{\partial \mathbf{n}}\right)^2|\mathbf{n}_G|^2((x,(1+s)y) \cdot \mathbf{n})\,d\sigma_z.
\end{equation}
This identity is a consequence of a more general class of Pohozev-type identities (see, {\it e.g.},  Tri \cite{Tr98A, Tr98B} for $N=2$, Kogoj 
and Lanconelli \cite[Section 2]{KoLa12} for arbitrary $N$). We also mention Garofalo and Lanconelli \cite{GaLa92} for a Pohozaev-type identity for the Heisenberg Laplacian.
 A proof of \eqref{eq:rp}  can be done following the classical 
argument that Rellich  used for the standard Laplacian in \cite{Re40}.  This argument  is rather elementary being  based only on  integration by parts, but requires some lengthy computations.
Instead, our proof is a straightforward application of the Hadamard-type formula that we have proved. More precisely, our strategy is  first to differentiate the eigenvalue with respect to the natural dilation in the Grushin setting, and then to match this derivative with the one computed by \eqref{Nhadf1}.
The same  strategy was exploited by di Blasio and Lamberti \cite[Theorem 5.1]{LaDi19} for the 
Finsler $p$-Laplacian.

 The natural dilation in the Grushin setting is:
\[
\delta_{t}(z) : = (tx,t^{1+s}y) \qquad \forall z=(x,y) \in \mathbb{R}^N, \forall t >0.
\]
We fix  $\Omega$ to be a bounded open subset of $\mathbb{R}^N$ of class $C^1$. We set 
\[
\Omega_t := \delta_{t}(\Omega) \qquad \forall  t>0.
\]
 Let $\lambda = \lambda_n[\Omega] = \cdots = \lambda_{n+m-1}[\Omega]$ be a Dirichlet Grushin eigenvalue on $\Omega$ of multiplicity $m$.
It is easily seen that 
\begin{equation}\label{eigenrescale}
t^2\lambda_{n+l}[\Omega_t] =\lambda \qquad \forall t>0,	\,\forall l = 0,\ldots,m-1\, .
\end{equation}
This can be deduced from the fact that if $l = 0,\ldots,m-1$ and  $u$ is an eigenfunction   in  $W^{1,2}_{G,0}(\Omega_t)$ corresponding to $\lambda_{n+l}[\Omega_t]$, then we have
\begin{align*}
\int_{\Omega}\nabla_G( u(\delta_t))& \cdot\nabla_G(\varphi(\delta_t)) \,dz \\
=&  \int_{\Omega}(\nabla( u(\delta_t))I_G(z)) \cdot (\nabla(\varphi(\delta_t))I_G(z)) \,dz\\
=&\int_{\Omega}\nabla_x( u(\delta_t))\cdot \nabla_x(\varphi(\delta_t))  \,dz +
\int_{\Omega}|x|^{2s}\nabla_y( u(\delta_t))\cdot \nabla_y(\varphi(\delta_t)) \,dz \\
=&\int_{\Omega}t^2\nabla_xu(\delta_t)\cdot \nabla_x\varphi(\delta_t) \,dz +
\int_{\Omega}|x|^{2s}t^{2+2s}\nabla_y u(\delta_t)\cdot \nabla_y\varphi(\delta_t) \,dz\\
=&\int_{\Omega_t}t^2\nabla_xu\cdot \nabla_x\varphi t^{-h-(1+s)k} \,dz +
\int_{\Omega_t}|x|^{2s}t^{2}\nabla_y u\cdot \nabla_y\varphi t^{-h-(1+s)k}\,dz\\
=&\int_{\Omega_t}t^2\nabla_Gu\cdot \nabla_G\varphi t^{-h-(1+s)k}\,dz\\
=&t^2\lambda_{n+l}[\Omega_t]\int_{\Omega_t} u\varphi t^{-h-(1+s)k}\,dz\\
=&t^2\lambda_{n+l}[\Omega_t]\int_{\Omega} u(\delta_t)\varphi(\delta_t)\,dz \qquad \forall 
\varphi \in W_{G,0}^{1,2}(\Omega_t).
\end{align*}
Accordingly, by \eqref{eigenrescale} with $t=1+\varepsilon$, we have
\[
\lambda_{n+l}[\Omega_{1+\varepsilon}] =(1+\varepsilon)^{-2}\lambda \qquad \forall \varepsilon>-1,\,	\forall l = 0,\ldots,m-1\,  .
\]
Therefore, $\varepsilon \mapsto \lambda_{n}[\Omega_{1+\varepsilon}]$ is differentiable in $\mathopen]-1,+\infty[$ and we have
\[
\frac{d}{d\varepsilon}\lambda_{n}[\Omega_{1+\varepsilon}]\bigg|_{\varepsilon=0} = -2\lambda.
\] 
On the other side, we can also compute the derivative  $\frac{d}{d\varepsilon}\lambda_n[\Omega_{1+\varepsilon}]\Big|_{\varepsilon=0}$ exploiting  our results. Let $u$ be an 
eigenfunction corresponding to $\lambda$ normalized such that $\|u\|_{L^2(\Omega)}=1$, and assume that
$u \in W^{1,2}_0(\Omega) \cap W^{2,2}(\Omega)$.  We note that  if $O$ is any bounded open subset of 
$\mathbb{R}^N$ containing $\overline \Omega \cap \{x=0\}$, then
$\delta_{1+\varepsilon} \in \mathcal{A}_{\Omega, O}$ for all  $\varepsilon > -1$. We can apply Theorem \ref{RN} 
and Remark \ref{rem:normaliz} to the family $\{\delta_{1+\varepsilon}\}$ and 
 obtain that the eigenvalues of the matrix $\eqref{RN1B}$  are the derivatives at $\varepsilon = 0$ of the branches 
splitting from $\lambda$. As we have already seen above, the domain perturbation 
$\delta_{1+\varepsilon}$ preserves the multiplicity of the eigenvalue $\lambda$ and accordingly the matrix 
$\eqref{RN1B}$ is actually a scalar matrix.
 Thus
\begin{align*}
\frac{d}{d\varepsilon}\lambda_n[\Omega_{1+\varepsilon}]\bigg|_{\varepsilon=0} =&  - \int_{\partial\Omega}
  \left(\frac{\partial u}{\partial \mathbf{n}}\right)^2|\mathbf{n}_G|^2((x,(1+s)y) \cdot \mathbf{n})\,d\sigma_z.
\end{align*}
By the above two expressions of the derivative of the the eigenvalue we get the Rellich-Pohozaev identity:
\[
\lambda = \frac{1}{2}\int_{\partial\Omega}
  \left(\frac{\partial u}{\partial \mathbf{n}}\right)^2|\mathbf{n}_G|^2((x,(1+s)y) \cdot \mathbf{n})\,d\sigma_z.
\]

\section*{Acknowledgments} The authors are members of the `Gruppo Nazionale per l'Analisi Matematica, la Probabilit\`a e le loro Applicazioni' (GNAMPA) of the `Istituto Nazionale di Alta Matematica' (INdAM) and acknowledge the support of the Project BIRD191739/19 `Sensitivity analysis of partial differential equations in
the mathematical theory of electromagnetism' of the University of Padova.  
P.~Luzzini and P.~Musolino acknowledge the support of the  `INdAM GNAMPA Project 2020 - Analisi e ottimizzazione asintotica per autovalori in domini con piccoli buchi'. P.~Musolino also acknowledges the support of  the grant `Challenges in Asymptotic and Shape Analysis - CASA'  of the Ca' Foscari University of Venice.  The authors are very thankful to Prof. Enrique Zuazua for bringing to their attention the method which allows to deduce Rellich-type identities from Hadamard-type formulas.


\begin{thebibliography}{11}

\bibitem{ArDa08}
W. Arendt, D. Daners, 
Varying domains: stability of the Dirichlet and the Poisson problem.
Discrete Contin. Dyn. Syst. 21 (2008), no. 1, 21--39.

\bibitem{Ar95}
J. M. Arrieta, Neumann eigenvalue problems on exterior perturbations of the domain, J. Differential Equations, 118 (1995), no. 1, 54--103.

\bibitem{ArCa04}
J. M. Arrieta, A. N. Carvalho, Spectral convergence and nonlinear dynamics of reaction-diffusion equations under perturbations of the domain, J. Differential Equations, 199 (2004), no. 1, 143-178.


\bibitem{AsBe95}
M.S. Ashbaugh, R.D. Benguria, On Rayleigh's conjecture for the clamped plate and its generalization to three dimensions, Duke Math. J., 78 (1995), no. 1, 1-7.




\bibitem{Ba67}
M. S. Baouendi, Sur une classe d'operateurs elliptiques degeneres, Bull. Soc. Math. France 95 (1967), 45--87.



\bibitem{BuBu05}
D. Bucur, G. Buttazzo, Variational methods in shape optimization problems. Progress in Nonlinear Differential Equations and their Applications, 65. Birkh\"auser Boston, Inc., Boston, MA, 2005. 

\bibitem{BuLa13}
D. Buoso, P.D. Lamberti, Eigenvalues of polyharmonic operators on variable domains, ESAIM Control Optim. Calc. Var. 19 (2013), no. 4, 1225--1235.




\bibitem{BuLa15}
D. Buoso, P.D. Lamberti,
Shape sensitivity analysis of the eigenvalues of the Reissner-Mindlin system, SIAM J. Math. Anal. 47 (2015), no. 1, 407--426. 


\bibitem{BuPr15}
D. Buoso,  L. Provenzano, 
A few shape optimization results for a biharmonic Steklov problem. 
J. Differential Equations 259 (2015), no. 5, 1778--1818. 

\bibitem{BuLaLa06}
 V.I. Burenkov, P.D. Lamberti, M. Lanza de Cristoforis, Spectral stability of nonnegative selfadjoint operators, Sovrem. Mat. Fundam. Napravl., 15 (2006), 76--111.

 \bibitem{Da04}
 L. D'Ambrosio, Hardy inequalities related to Grushin type operators, 
 Proc. Amer. Math. Soc. 132 (2004), no. 3, 725--734.
 
 \bibitem{Da08}
 D. Daners, Domain perturbation for linear and semi-linear bound- ary value problems, Handbook of differential equations: stationary partial differential equations, Vol. VI, 1-81, Handb. Differ. Equ., Elsevier/North-Holland, Amsterdam, 2008.

\bibitem{Da95}
 E. B. Davies, Spectral theory and differential operators,  Cambridge Studies in Advanced Mathematics, 42. Cambridge University Press, Cambridge, 1995.



 \bibitem{DeZo11}  
 M. C.~Delfour, J.P.~Zol\'esio, Shapes and geometries. Metrics, analysis, differential calculus, and optimization. Second edition. Advances in Design and Control, 22. Society for Industrial and Applied Mathematics (SIAM), Philadelphia, PA, 2011. 
 

\bibitem{LaDi19}
G. di Blasio, P. D. Lamberti,
Eigenvalues of the Finsler p-Laplacian on varying domains,  Mathematika 66 (2020), no. 3, 765--776.

\bibitem{Fa23}
G. Faber, Beweis, dass unter allen homogenen Membranen von gleicher Fl\"ache und gleicher Spannung die
 kreisf\"ormige den tiefsten Grundton gibt, Sitz. Ber. Bayer. Akad. Wiss. (1923), 169--72.



\bibitem{FaWe19}
 M. M. Fall, T. Weth, Critical domains for the first nonzero Neumann eigenvalue in Riemannian manifolds. J. Geom. Anal. 29 (2019), no. 4, 3221--3247. 

\bibitem{FrLa82}
 B. Franchi, E. Lanconelli, Une m\'etrique associ\'ee \`a une classe d'op\'erateurs elliptiques d\'eg\'ener\'es,  Rend. Sem. Mat. Univ. Politec. Torino 1983, Special Issue, 105--114 (1984). 
 
 
 \bibitem{FrLa83}
 B. Franchi, E. Lanconelli, H\"older regularity theorem for a class of linear nonuniformly elliptic operators with measurable coefficients, Ann. Scuola Norm. Sup. Pisa Cl. Sci. (4) 10 (1983), no. 4, 523--541. 
 
  \bibitem{FrLa84}
  B. Franchi, E. Lanconelli,   An embedding theorem for Sobolev spaces related to nonsmooth vector fields and Harnack inequality, Comm. Partial Differential Equations 9 (1984), no. 13, 1237--1264. 

\bibitem{FrSe87}
B. Franchi, R. Serapioni, Pointwise estimates for a class of strongly degenerate elliptic operators: a geometrical approach,
Ann. Sc. Norm. Super. Pisa Cl. Sci. (4) 14 (1987), no. 4, 527--568 (1988).


\bibitem{FrSeSe96}
B. Franchi, R. Serapioni, F. Serra Cassano,
Meyers-Serrin type theorems and relaxation of variational integrals depending on vector fields.
Houston J. Math. 22 (1996), no. 4, 859--890.



\bibitem{GaLa92}
 N. Garofalo, E. Lanconelli, Existence and nonexistence results for semilinear equations on the Heisenberg group. Indiana Univ. Math. J. 41 (1992), no. 1, 71--98. 
 
 
\bibitem{GaSh94}
N.~Garofalo, Z.~Shen, Carleman estimates for a subelliptic operator and unique continuation. Ann. Inst. Fourier (Grenoble) 44 (1994), no. 1, 129--166.



 \bibitem{Gr70}
V.V. Gru\v{s}in, A certain class of hypoelliptic operators, Mat. Sb. (N.S.) 83 (125) 1970 456--473. 

\bibitem{Gr71}
V. V. Gru\v{s}in,  A certain class of elliptic pseudodifferential operators that are degenerate on a submanifold, Mat. Sb. (N.S.) 84 (126) 1971 163--195. 


\bibitem{He06}
A. Henrot, Extremum problems for eigenvalues of elliptic operators, Frontiers in Mathematics, Birkh\"auser Verlag, Basel, 2006.

  \bibitem{HePi05}
A.~Henrot, M.~Pierre, Variation et optimisation de formes, Vol.~48 of
  Math\'ematiques \& Applications (Berlin) [Mathematics \& Applications],
  Springer, Berlin, 2005. 
  

  \bibitem{He82}
 D. Henry, Topics in nonlinear analysis. Universidade de Brasilia, Trabalho de Matematica  192, 1982.
  
  \bibitem{Je81}
D.S. Jerison,
The Dirichlet problem for the Kohn Laplacian on the Heisenberg group. II. 
J. Funct. Anal. 43, 224--257 (1981). 
  

\bibitem{KoLa09}
A. E. Kogoj, E. Lanconelli, Liouville theorem for X-elliptic operators, Nonlinear Anal. 70 (2009), no. 8, 2974--2985.

\bibitem{KoLa12}
 A. E. Kogoj, E. Lanconelli, On semilinear $\Delta_\lambda$-Laplace equation, Nonlinear Anal. 75 (2012), no. 12, 4637--4649. 
 
 \bibitem{KoLa18}
 A. E. Kogoj, E. Lanconelli, 
Linear and semilinear problems involving $\Delta_\lambda$-Laplacians. Proceedings of the International Conference "Two nonlinear days in Urbino 2017", 167--178, Electron. J. Differ. Equ. Conf., 25, Texas State Univ. San Marcos, Dept. Math., San Marcos, TX, 2018.

\bibitem{Kr24}
E. Krahn, \"Uber eine von Rayleigh formulierte Minimaleigenschaft des Kreises, Math. Ann., 94 (1924), 97-100.

\bibitem{KoNi65}
J.J. Kohn, L. Nirenberg,
Non-coercive boundary value problems. 
Commun. Pure Appl. Math. 18, 443--492 (1965). 

\bibitem{La14}
P. D. Lamberti,
 Steklov-type eigenvalues associated with best Sobolev trace constants: domain perturbation and overdetermined systems. Complex Var. Elliptic Equ. 59 (2014), no. 3, 309--323.

\bibitem{La09}
P. D. Lamberti, Absence of critical mass densities for a vibrating membrane, Appl. Math. Optim. 59 (2009), 319--327.



\bibitem{LaLa04} 
P. D. Lamberti, M. Lanza de Cristoforis,  A real analyticity result for symmetric functions of the eigenvalues of a domain dependent Dirichlet problem for the Laplace operator, J. Nonlinear Convex Anal. 5 (2004), no. 1, 19--42.

\bibitem{LaLa06}
P. D. Lamberti, M. Lanza de Cristoforis, Critical points of the symmetric functions of the eigenvalues of the Laplace operator and overdetermined problems, J. Math. Soc. Japan 58 (2006), no. 1, 231--245. 

\bibitem{LaPr13}
P. D. Lamberti, L. Provenzano, A maximum principle in spectral optimization problems for elliptic operators subject to mass density perturbations,  Eurasian Math. J., 4 (2013), no. 3, 70--83 .

\bibitem{Mo06} R.~Monti,  Sobolev inequalities for weighted gradients. Comm. Partial Differential Equations 31 (2006), no. 10-12, 1479--1504.


\bibitem{Mo10}
 D.D. Monticelli, Maximum principles and the method of moving planes for a class of degenerate elliptic linear operators. J. Eur. Math. Soc. (JEMS) 12 (2010), no. 3, 611--654.
 

\bibitem{MoPa09}
D.D. Monticelli, K. R. Payne,
Maximum principles for weak solutions of degenerate elliptic equations with a uniformly elliptic direction,
J. Differential Equations 247 (2009), no. 7,  1993--2026.

\bibitem{MoPaPu19}
D.D. Monticelli, K.R. Payne, F. Punzo, Poincar\'e inequalities for Sobolev spaces with matrix-valued weights and applications to degenerate partial differential equations, Proc. Roy. Soc. Edinburgh Sect. A 149 (2019), no. 1, 61--100.

\bibitem{Na95}
N. S. Nadirashvili, Rayleigh's conjecture on the principal frequency of the clamped plate, Arch. Rational Mech. Anal., 129 (1995), no. 1, 1--10.




 \bibitem{NoSo13}
A.A.~Novotny, J.~Soko\l owski,  Topological derivatives in shape optimization, Interaction of Mechanics and Mathematics, Springer, Heidelberg, 2013.
 
 
 
\bibitem{NoSoZo19}  
A.A.~Novotny, J.~Soko\l owski, A. \.{Z}ochowski,  Applications of the topological derivative method. With a foreword by Michel Delfour, Studies in Systems, Decision and Control, 188. Springer, Cham, 2019.




\bibitem{Pi84}   
O. Pironneau, Optimal shape design for elliptic systems. Springer Series in Computational Physics. Springer-Verlag, New York, 1984.



\bibitem{Pr94}
G. Prodi, Dipendenza dal dominio degli autovalori dell'operatore di Laplace, Istituto Lombardo,
Rend. Sc., 128 (1994), pp. 3--18.


 \bibitem{Re40}
F. Rellich, 
Darstellung der Eigenwerte von $\Delta u + \lambda u=0$ durch ein Randintegral.
 Math. Z. 46 (1940), 635--636. 


\bibitem{Re69} F.~Rellich,  Perturbation theory of eigenvalue problems, Gordon and Breach Science Publ., New York  (1969).

\bibitem{SoZo92}
J.~Soko{\l}{o}wski, J.-P. Zol\'esio, Introduction to shape optimization.
  {S}hape sensitivity analysis, Vol.~16 of Springer Series in Computational
  Mathematics, Springer-Verlag, Berlin, 1992.

\bibitem{ThTr02}
 N.T.C. Thuy, N.M. Tri, Some existence and nonexistence results for boundary value problems for semilinear elliptic degenerate operators, Russ. J. Math. Phys. 9 (2002), no. 3, 365--370.
 
 
\bibitem{ThTr12}
 P.T. Thuy, N.M. Tri, Nontrivial solutions to boundary value problems for semilinear strongly degenerate elliptic differential equations, NoDEA Nonlinear Differential Equations Appl. 19 (2012), no. 3, 279--298. 
 
 


\bibitem{Tr98A}
N.M. Tri, On the Grushin equation, Mat. Zametki 63 (1998), no. 1, 95--105; translation in Math. Notes 63 (1998), no. 1-2, 84--93.

\bibitem{Tr98B}
N.M. Tri, Critical Sobolev exponent for degenerate elliptic operators, Acta Math. Vietnam. 23 (1998), no. 1, 83--94. 

\bibitem{Tr09}
 N.M. Tri, Recent results in the theory of semilinear elliptic degenerate differential equations,
Vietnam J. Math. 37 (2009), no. 2-3, 387--397.



\end{thebibliography}
\end{document}